\documentclass[11pt]{article}
\usepackage{latexsym,amssymb,amsmath,amsfonts,amsthm}
\usepackage{enumerate}
\usepackage{cases}
\usepackage{graphics}
\usepackage{graphicx}
\usepackage{mathrsfs}
\usepackage{color}
\usepackage{bm}
\usepackage{subfigure}
\usepackage{paralist}

\usepackage{hyperref}%参考文献回引[backref]
\hypersetup{hidelinks}

\makeatletter
\@addtoreset{equation}{section}
\makeatother %section计数器增加时，equation计数器自动置零

\newcommand{\R}{{\mat R}}

\newcommand{\Ct}{{\mat C}}
\newcommand{\E}{{\mat E}}

\newcommand{\ds}{\displaystyle}
\newcommand{\no}{\nonumber}
\newcommand{\be}{\begin{equation}\begin{aligned}}
\newcommand{\ben}{\begin{equation*}\begin{aligned}}
\newcommand{\en}{\end{aligned}\end{equation}}
\newcommand{\enn}{\end{aligned}\end{equation*}}
\newcommand{\ba}{\backslash}
\newcommand{\pa}{\partial}

\newcommand{\ov}{\overline}
\newcommand{\I}{{\rm Im}}
\newcommand{\Rt}{{\rm Re}}
\newcommand{\curl}{{\rm curl}}
\newcommand{\Curl}{{\rm Curl}}

\newcommand{\dive}{{\rm div}}
\newcommand{\Dive}{{\rm Div}}
\newcommand{\g}{\gamma}
\newcommand{\Ga}{\Gamma}
\newcommand{\eps}{\epsilon}
\newcommand{\vep}{\varepsilon}
\newcommand{\Om}{\Omega}
\newcommand{\om}{\omega}
\newcommand{\sig}{\sigma}
\newcommand{\Sig}{\Sigma}

\newcommand{\bt}{\beta}
\newcommand{\ti}{\times}

\newcommand{\wit}{\widetilde}
\newcommand{\wih}{\widehat}

\newcommand{\De}{\Delta}
\newcommand{\ra}{\rightarrow}

\newcommand{\na}{\nabla}
\newcommand{\mat}{\mathbb}

\newcommand{\ify}{\infty}

\newcommand{\la}{\lambda}

\newcommand{\les}{\lesssim}
\newcommand{\ch}{\check}
\newcommand{\V}{\Vert}
\newcommand{\half}{\frac{1}{2}}

\newcommand{\n}{\bm{n}}

\newcommand{\0}{\bm{0}}
\newcommand{\for}{{\rm for}}
\newcommand{\on}{{\rm on}}
\newcommand{\gin}{{\rm in}}
\newcommand{\gand}{{\rm and}}
\newcommand{\diag}{{\rm diag}}
\newtheorem{theorem}{Theorem}[section]
\newtheorem{lem}[theorem]{Lemma}

\newtheorem{remark}[theorem]{Remark}

\begin{document}
\renewcommand{\theequation}{\arabic{section}.\arabic{equation}} %公式编号与章节关联

\title{\bf
% Convergence of the PML method for time-dependent electromagnetic scattering 
% by an elastic body in a two-layered medium
Time domain analysis for electromagnetic scattering
by an elastic obstacle in a two-layered medium
}
\author{Changkun Wei\thanks{Research Institute of Mathematics,
 Seoul National University, Seoul, 08826, Republic of Korea ({\tt
ckun.wei@snu.ac.kr})}
\and
Jiaqing Yang\thanks{School of Mathematics and Statistics, Xi'an Jiaotong University,
Xi'an, Shaanxi, 710049, China ({\tt jiaq.yang@xjtu.edu.cn})}
\and
Bo Zhang\thanks{NCMIS, LSEC and Academy of Mathematics and Systems Science, Chinese Academy of Sciences,
Beijing, 100190, China and School of Mathematical Sciences, University of Chinese Academy of Sciences,
Beijing 100049, China ({\tt b.zhang@amt.ac.cn})}
}
\date{}

\maketitle

%\vspace{.2in}

\begin{abstract}
In this paper, we consider the scattering of a time-dependent electromagnetic wave by an elastic body immersed in the lower half-space of a two-layered background medium which is separated by an unbounded rough surface. By proposing two exact transparent boundary conditions (TBCs) on the artificial planes, we reformulate the unbounded scattering problem into an equivalent 
initial-boundary value problem in a strip domain with the well-posedness and stability proved using the Laplace transform, variational method and energy method. A perfectly matched layer (PML) is then introduced to truncate the interaction problem with two finite layers containing the elastic body, leading to a PML problem in a finite strip domain. We further verify the existence, 
uniqueness and stability estimate of solution for the PML problem. Finally, we establish the exponential convergence in terms of the thickness and parameters of the PML layers via an error estimate on the electric-to-magnetic (EtM) capacity operators between the original problem and the PML problem.

\vspace{.2in}
{\bf Keywords:} Electromagnetic wave equation, elastic wave equation, two-layered medium, 
time-domain, well-posedness, perfectly matched layer, exponential convergence.
\end{abstract}

\section{Introduction}
Let us consider the interaction scattering of a time-dependent electromagnetic field by an elastic body embedded in a two-layered medium in three dimensions. This problem can be categorized into the
class of the unbounded rough surface scattering problems, which are the subject of intensive studies in the engineering and mathematics. In the problem setting, the whole space is divided into two parts by an unbounded rough surface $\Ga_f$ with the elastic body $\Om$ immersed in the lower half-space. We assume that the electromagnetic field initiated by an electric current density produces a 
tangential stress on the interface $\Ga:=\pa\Om$ which excites an elastic displacement of the elastomer. Following the Voigt's model (cf. \cite{Maugin1988,Cakoni2003,BMM2010,GHM2010}), we assume 
that the electromagnetic field does not considerably penetrate inside the elastomer. Several important works have been done on this typical electromagnetic-elastic interaction problem,
which is confined to the time-harmonic setting.  It was shown in \cite{Cakoni2003} that Cakoni \& Hsiao established a mathematical model, for which the uniqueness and an equivalent
boundary-field equation formulation as well as a weak variational formulation were presented in an appropriate Sobolev space. Based on the framework of \cite{Cakoni2003}, the existence of a solution was
shown by using the variational method \cite{GHM2010}, which was later extended to a different Sobolev space for the elastic field \cite{BMM2010}.
Further, it was shown in \cite{GHM2010} that a finite element Galerkin scheme was provided to compute both the scattered electromagnetic field and the elastic displacement. Very recently, the well-posedness
was established for the interaction problem in \cite{zyz2019} with general transmission conditions via the variational method in combination with the classical Fredholm alternative.

In this paper, we aim to present a theoretical analysis for the time-dependent electromagnetic scattering by an elastic body in a two-layered medium. The goal of this work consists of the following three parts:
\begin{itemize}
	\item Prove the well-posedness and stability for the interaction problem;
	\item Propose a time-domain PML method and show the well-posedness
	and stability;
	\item Establish the exponential convergence of the PML method in terms of thickness and parameters 
	of the PML layer.
\end{itemize}

Due to the unbounded interface, the usual Silver-M$\ddot{u}$ller radiation condition is not valid anymore to describe the asymptotic behavior of scattered waves away from the rough surface. 
Moreover, the classical Fredholm alternative theorem may not be applied into this kind of problems due to the lack of compactness result. These make the studies of interface scattering problems quite 
challenging. For the time-harmonic setting, there exists lots of works for the mathematical analysis with using either the boundary integral equations method or the variational method; see, e.g.  
\cite{CB1998,CHP2006,CM2005,CRZ1999,Zhang2003} for the acoustic wave and \cite{HL2011,LWZ2011,LZZ2017} for the electromagnetic wave. Recently, the time-domain scattering problems have 
attracted much attention due to their capability of capturing wide-band signals and modeling more general material and nonlinearity \cite{CM2014,Li2012,WW2014,Wang2012,Zhao2013}. Precisely, the mathematical analysis can be found in \cite{chen2009,WW2014} for time-dependent scattering problems in the full acoustic wave cases, and \cite{chen2008,LWW2015,GL2016,GL2017} in the full 
electromagnetic wave cases. In addition, the time-dependent fluid-solid interaction problems has been also studied for the bounded elastomer \cite{Bao2018}, local rough surfaces \cite{wy2019}, and unbounded layered structures in the three-dimensional case \cite{GLZ2017}. To the best of our knowledge, the mathematical analysis is quite rare for the electromagnetic-elastic interaction problems in the 
time domain. Here, we refer to a recently related work \cite{wyz-amas-2020} for a bounded obstacle embedded in the homogeneous background medium.

As is known, the perfectly matched layer (PML) method is a fast and effective method for solving unbounded scattering problems which was originally proposed by B\'erenger in 1994 for Maxwell's equations 
\cite{BERENGER1994}.  The purpose of the PML method is to surround the computational domain with a specially designed medium in a finite thickness layer in which the scattered waves decay
rapidly regardless of the wave incident angle, thereby greatly reducing the computational complexity of the scattering problem. Since then, various PML formulations have been widely created and 
studied for solving the wave scattering problems (see, e.g., \cite{TURKEL1998,Collino1998,Lassas1998,Teixeira2001,CW2003,CM2009,chen2009}).
The broad applications of the PML method attract great interests for mathematicians to study the convergence analysis for the time-harmonic scattering problems; see, e.g. 
\cite{Lassas1998,HSZ2003,CZ2010,BP2013,BW2005,Bramble2007,BP2008,BP2012,CZ2017} for the acoustic and electromagnetic obstacle scattering problems. However, the PML technique is much less studied for unbounded rough surface scattering problems. A general linear convergence was proved in \cite{CM2009} for the acoustic scattering problem depending on the thickness and 
composition of the layer. Moreover, an exponential convergence was also established in \cite{LWZ2011} for the electromagnetic scattering problems.

Compared with the time-harmonic setting, very few results are available for the mathematical analysis of the time-domain PML method, which is challenged by the dependence of the absorbing medium on 
all frequencies. For the 2D time-domain acoustic scattering problem, the exponential convergence of both a circular PML method \cite{chen2009} and a uniaxial PML method \cite{Chen2012} were 
established in terms of the thickness and absorbing parameters. For the 3D time domain electromagnetic scattering problem,
the exponential convergence of a spherical PML method was very recently shown in \cite{wyz2019} in terms of the thickness and parameter of the PML layer, based on a real coordinate stretching 
technique associated with $[\Rt(s)]^{-1}$ in the Laplace domain, where $s\in\Ct_+$ is the Laplace transform variable. It is also noticed that for the acoustic-elastic interaction problem, the well-posedness 
and stability estimates of the time-domain PML method was proved in \cite{Bao2018}, but no convergence analysis was provided. We also remark that an exponential convergence of the PML method was 
recently established in \cite{wyz-cms-2020} for the fluid-solid interaction problem above an unbounded rough surface, which generalized our previous idea \cite{wyz2019} with the real coordinate stretching technique.

In this paper, we intend to study the time-dependent electromagnetic-elastic interaction problem in a two-layered medium associated with a bounded elastic body immersed in the lower half-space. With 
the aid of the factorization on the interface conditions \cite{wyz-amas-2020} and two exact time domain TBCs, we establish the well-posedness and stability of the interaction problem based on the variational method and the Laplace transform and its inversion. Further, we propose a time-domain PML method along $x_3$ direction by using the real coordinate stretching technique in \cite{wyz2019} associated 
with $[\Rt(s)]^{-1}$ in the frequency domain.  The well-posedness and stability estimate of the truncated PML problem are proved by the Laplace transform and energy method. An exponential convergence 
is then proved in terms of the thickness and parameters of the PML layer, through an error estimate on the EtM operators between the original problem and the PML problem.

The outline of this paper is as follows. In section \ref{sec-model}, we introduce some basic notations and give a brief description of our model problem. In section \ref{sec-well_posed}, the original interaction scattering problem is firstly reduced into an equivalent initial-boundary value problem in a strip domain. Then we study the well-posedness and stability for the reduced problem by the variational method and 
the energy method. In section \ref{sec-pml}, a time-domain PML method is introduced to truncate the interaction problem with two finite layers containing the elastic body, leading to a truncated PML problem 
in a finite strip domain. The well-posedness and stability estimate for the truncated PML problem is further verified. An exponential convergence of the PML method is finally established. Some conclusions 
are given in section \ref{sec6}.

\section{Problem formulation}\label{sec-model}

\begin{figure}[!htbp]
\setcounter{subfigure}{0}
  \centering
  \includegraphics[width=3.5in]{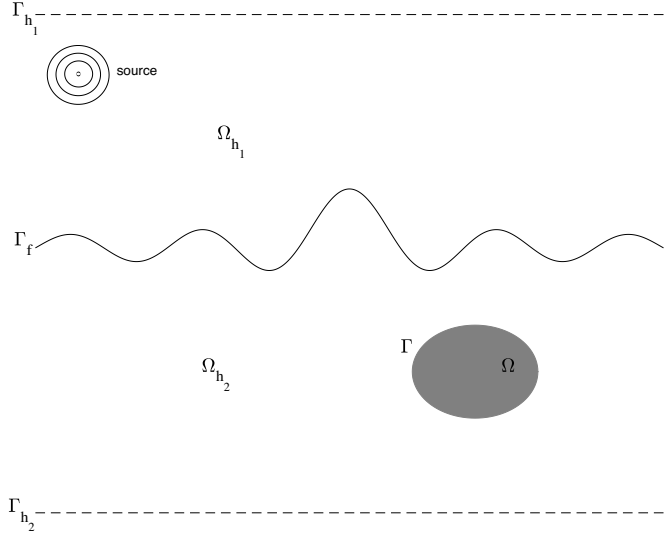}
 \caption{Geometric configuration of the scattering problem}\label{geometry}
\end{figure}

Consider the propagation of an electromagnetic wave  which is excited by an electric current density in a two-layered medium with a bounded elastic body immersed in the lower half-space; see 
the problem geometry in Figure \ref{geometry}. For $x=(x_1,x_2,x_3)^{\top}\in\R^3$, let $\wit{x}=(x_1,x_2)^{\top}\in\R^2$ and
\ben
\Ga_f:=\{x\in\R^3:x_3=f(\wit{x})\}
\enn
be the unbounded rough surface with $f\in C^2(\R^2)$, which 
separates the whole space into a two-layered domain
\ben
\Om_f^+:=\{x\in\R^3:x_3>f(\wit{x})\}\;\;\gand\;\;\Om_f^-:=\{x\in\R^3:x_3<f(\wit{x})\}.
\enn
Here, the electromagnetic medium fills with distinct parameters $\vep, \mu$. We assume that $\Om$ is a bounded domain with Lipschitz-continuous boundary $\Ga:=\pa\Om$ representing a homogeneous 
and isotropic elastic body immersed in the lower medium $\Om_f^-$ and the exterior $\Om^c = \R^3\ba \ov{\Om}$ of $\Om$ is simply connected. Furthermore, we assume $\Om$ to be with a constant mass density $\rho_i>0$, and Lam\'{e} constants $\la_i$, $\mu_i$ satisfying the condition that $\mu_i > 0$ and $3\la_i + 2\mu_i > 0$. Define two artificial planar surfaces
$\Ga_{h_1}:=\{x\in\R^3: x_3=h_1\}$, where $h_1>\sup_{\wit{x}\in\R^2}f(\wit{x})$ is a constant and $\Ga_{h_2}:=\{x\in\R^3: x_3=h_2\}$, where $h_2<0$ is small enough such that $\Om$ is
over plane $\Ga_{h_2}$. Let $\Om_{h_1}:=\{x\in\R^3: f(\wit{x})<x_3<h_1\}$ and $\Om_{h_2}:=\{x\in\R^3: h_2<x_3<f(\wit{x})\}\cap \Om^c$, and $\Om_{h}=\Om_{h_1}\cup \Om_{h_2}\cup\Ga_f$.
In what follows, we denote by $\n$ the unit outward normal vector both on $\Ga$  and $\Ga_f$ as well as $\n_1=(0,0,1)^{\top},\;\n_2 = (0,0,-1)^{\top}$  the unit outward normal vectors on $\Ga_{h_1}$ and 
$\Ga_{h_2}$, respectively. To the end, we define $\Ct_+:=\{s=s_1+is_2\in\Ct$ with $s_1,s_2\in\R$ and $s_1>0\}$ and remark hereafter that the index $j$ is always valued in the set $\{1,2\}$ except special statement.\\

{\bf Elastic wave equation}. In the elastic body $\Om$, the elastic displacement $\bm u$ is governed by the linear elastodynamic equation:
\be\label{2.1}
    \rho_i\frac{\pa^{2}\bm u}{\pa t^{2}}-\De^*\bm u=\0,\;\; \gin\;\;\Om\ti(0,T)
\en
where $\De^*$ is the Lam\'{e} operator defined as
\ben
	\De^*\bm u
	:= \mu_i\De\bm u+(\la_i+\mu_i)\na\dive\bm u
	 = \dive\bm\sig(\bm u).
\enn
In above,
$\bm \sig(\bm u)$ and $\bm \vep(\bm u)$ are called stress and strain tensors respectively, which are given by
\ben
	\bm \sig(\bm u)=(\la_i \dive\bm u)\bm{I}+2\mu_i\bm\vep(\bm u)\qquad\gand\qquad
	\bm\vep(\bm u)=\half(\na\bm u+(\na\bm u)^{\top}).
\enn
Furthermore,  the homogeneous initial conditions are imposed for the elastic wave equation
\be\label{2.1a}
	\bm{u}(x,0)=\frac{\pa\bm{u}}{\pa t}(x,0)=0,\;\; x\in\Om.
\en

{\bf Maxwell's equations}. In the electromagnetic domain $\Om^c$, the electric field $\bm{E}$ and magnetic field $\bm{H}$ satisfy the time-domain Maxwell equations
\be\label{2.2}
    \na\ti \bm{E}+\mu\frac{\pa \bm{H}}{\pa t}=\0,\;\;
    \na\ti \bm{H}-\vep\frac{\pa \bm{E}}{\pa t}=\bm J,\;\;\gin\;\;\Om^c\ti(0,T)
\en
where $\bm J$ is the electric current density which is assumed to be compactly supported in $\Om_h$ and $\bm J|_{t=0}=\0$, the electric permittivity $\vep$ and magnetic permeability $\mu$
are both positive and piece-wise constants:
\be\label{2.2a}
	\vep(x)=\begin{cases}
	     \vep_1,\quad x\in\Om_f^+,\\
	     \vep_2,\quad x\in\Om_f^-\ba\ov{\Om},
	     \end{cases}\quad
	\mu=\begin{cases}
	     \mu_1,\quad x\in\Om_f^+,\\
	     \mu_2,\quad x\in\Om_f^-\ba\ov{\Om}.
	     \end{cases}
\en
On the interface $\Ga_f$ between the two-layered medium, we have the jump conditions
\be\label{jump}
	\n\ti[\bm E]=\n\ti[\mu^{-1}\na\ti\bm E]=\0,\;\;&\on\;\;\Ga_f\ti (0,T)
\en
where $[\cdot]$ stands for the jump of a function across the interface $\Ga_f$. In addition, the homogeneous initial conditions are also imposed for the Maxwell's equations:
\be\label{2.2b}
    \bm{E}(x,0)=\bm{H}(x,0)=\0,\ x\in \Om^c.
\en
Using the Maxwell's system (\ref{2.2}), it is obvious that
\begin{align}\label{2.2c}
    \pa_t\bm E(x,0)=\vep^{-1}(\na\ti \bm{H})(x,0)-\vep^{-1}\bm J(x,0)=\0,&\ x\in \Om^c,\\
    \pa_t\bm H(x,0)=-\mu^{-1}(\na\ti \bm{E})(x,0)=\0,&\ x\in \Om^c.\label{2.2d}
\end{align}
and
\be
    \na\cdot\bm{E}=\na\cdot\bm{H}=0,\;\;\gin\;\Om^c\ti(0,T).
\en
Due to the unbounded structure of the medium, it is no longer valid to impose the usual Silver-M\"uler radiation condition. Instead, we employ the following radiation condition:
the electromagnetic fields $(\bm E,\bm H)$ consist of bounded outgoing waves in $\Om_{h_1}^+$ and $\Om_{h_2}^-$, where $\Om_{h_1}^+=\{x\in\R^3:x_3>h_1\}$ and $\Om_{h_2}^-=\{x\in\R^3:x_3<h_2\}$.\\

{\bf Interface conditions}.
The two medium are coupled by the interface condition (cf. \cite{Cakoni2003}):
\be\label{2.3}
    \bm{H}(x,t)\ti\bm{E}(x,t)\cdot\n=\bm{T}\bm{u}(x,t)\cdot\bm{u}_t(x,t),\qquad\on\;\;\Ga\ti(0,T)
\en
where $\bm{Tu}:=2\mu_i\n\cdot\na\bm{u}+\la_i\n\na\cdot\bm{u}+\mu_i\n\ti(\na\ti\bm{u})$ denotes the elastic surface traction operator.

There are infinite many decomposition of above interface condition (\ref{2.3}). According to the Voigt's model \cite{Maugin1988}, the stress tensor is proportional to the magnetic field which leads to the 
following decomposition (see \cite{wyz-amas-2020})
\be\label{2.4}
    \n\ti\bm{H}=\bm{Tu}\quad \gand\quad \n\ti\bm{E}=\n\ti\bm{u}_t, \;\;\on\;\;\Ga\ti(0,T).
\en
Note that it is easily checked out that (\ref{2.4}) implies (\ref{2.3}).

\section{The well-posedness of scattering problem}\label{sec-well_posed}

In this section, we firstly introduce two exact time-domain transparent boundary conditions (TBCs) on the artificial plane surfaces to reformulate the scattering problem into an initial-boundary
value problem in a finite strip domain. Then, we will show the well-posedness for the reduced problem in $s$-domain by the method of
Laplace transform and the Lax-Milgram lemma. To the end, the existence, uniqueness, and stability for the reduced problem in the time domain shall be verified by using the abstract inversion theorem 
of the Laplace transform, and the energy argument.

\subsection{Transparent boundary conditions.}
In this subsection, we start by introducing two transparent boundary conditions (TBCs) on the artificial planar surfaces (cf. \cite{GL2017}):
\be\label{TBCs1}
    \mathscr{T}_j[\bm{E}_{\Ga_{h_j}}]=\bm{H}\ti\n_j,&\quad\on\;\;\Ga_{h_j}\ti(0,T),\quad j=1,2,
\en
which maps the tangential component of electric field $\bm E$ to the tangential trace of magnetic field
 $\bm H$ on $\Ga_{h_j}$. Then the time-dependent electromagnetic-elastic wave interaction problem 
can be reduced to an equivalent initial boundary value problem in the strip domain $\Om_h$:
\be\label{2.7}
    \begin{cases}
        \ds\rho_i\frac{\pa^{2}\bm u}{\pa t^{2}}-\De^*\bm u=\0,& \gin\;\; \Om\ti(0,T) \\
        \ds\na\ti \bm{E}+\mu \pa_t\bm{H}=\0,&\gin\;\;  \Om_h\ti(0,T) \\
        \ds\na\ti \bm{H}-\vep \pa_t\bm{E}=\bm J,&\gin\;\;  \Om_h\ti(0,T)\\
        % \na\cdot\bm{E}=\na\cdot\bm{H}=0,&\gin\;\;  \Om_h\ti(0,T) \\
        \ds\bm u(x,0)=\pa_t\bm u(x,0)=\0,& \gin\;\;  \Om\\
        \ds\bm{E}(x,0)=\bm{H}(x,0)=\0,&\gin\;\;  \Om_h\\
        \ds\n\ti[\bm E]=\n\ti[\mu^{-1}\na\ti\bm E]=\0,&\on\;\;\Ga_f\ti(0,T)\\
		% \ds\n\ti[\bm H]=\n\ti[\vep^{-1}\na\ti\bm H]=\0, &\on\;\;\Ga_f\ti(0,T),\\
        \ds\n\ti\bm{H}=\bm{Tu},\;\n\ti\bm{E}=\n\ti\bm{u}_t, &\on\;\; \Ga\ti(0,T)\\
        % \ds\n\ti\bm{E}=\n\ti\bm{u}_t,&\on\;\; \Ga\ti(0,T),\\
        \mathscr{T}_j[\bm{E}_{\Ga_{h_j}}]=\bm{H}\ti\n_j,&\on\;\; \Ga_{h_j}\ti(0,T),
        \quad j=1,2.
    \end{cases}
\en
Taking the Laplace transform of $(\ref{2.7})$ and employing (\ref{A.2}) together with initial conditions (\ref{2.1a}) and (\ref{2.2b}), we obtain the time harmonic electromagnetic-elastic interaction
 problem in s-domain:
\be\label{2.8}
    \begin{cases}
        \ds\De^*\ch{\bm u}-\rho_is^2\ch{\bm u}=\0,& \gin\;\; \Om\\
        \ds\na\ti \ch{\bm{E}}+\mu s \ch{\bm{H}}=\0,&\gin\;\;  \Om_h\\
        \ds\na\ti \ch{\bm{H}}-\vep s \ch{\bm{E}}=\ch{\bm J},&\gin\;\;  \Om_h\\
        \ds\n\ti[\ch{\bm E}]=\n\ti[\mu^{-1}\na\ti\ch{\bm E}]=\0,&\on\;\;\Ga_f\\
		% \n\ti[\ch{\bm H}]=\n\ti[\vep^{-1}\na\ti\ch{\bm H}]=\0, &\on\;\;\Ga_f\\
        \ds\n\ti\ch{\bm{H}}=\bm{T\ch{u}},\;\n\ti\ch{\bm{E}}=\n\ti s\ch{\bm{u}}&\on\;\; \Ga\\
        % \ds\n\ti\ch{\bm{E}}=\n\ti s\ch{\bm{u}},&\on\;\; \Ga\\
        \ds\mathscr{B}_j[\ch{\bm{E}}_{\Ga_{h_j}}]=\ch{\bm{H}}\ti\n_j,&\on\;\; \Ga_{h_j},\quad j=1,2,
    \end{cases}
\en
where $s\in\Ct^+$, and $\mathscr{B}_j$ is the electric-to-magnetic (EtM) capacity operators on $\Ga_{h_j}$ in s-domain satisfying $\mathscr{T}_j=\mathscr{L}^{-1}\circ\mathscr{B}_j\circ\mathscr{L}$.

In \cite{GL2017}, Y. Gao and P. Li derived the formulation of the EtM operators $\mathscr{B}_j$ and showed some of  important properties including boundness and coercivity.
 Here, we present the main results of TBCs in \cite{GL2017} without detailed proof. The explicit representations of EtM operators $\mathscr{B}_j$ take the following form:
for any tangential vector $\bm\om=(\om_1,\om_2,0)^{\top}\;\on\;\Gamma_{h_j}$, 
\be\label{dtn}
\mathscr{B}_j[\bm\om]=(v_1,v_2,0)^{\top},
\en
where
\ben
	\wih{v}_1
	%&=(\mu_js)^{-1}[-\frac{\xi_1}{\beta_j(\xi)}(\xi_1\wih{\om}_1+\xi_2\wih{\om}_2)
	%+\beta_j(\xi)\wih{\om}_1]\\
	&=\frac{1}{\mu_js\beta_j(\xi)}\left[\vep_j\mu_js^2\wih{\om}_1
	+\xi_2(\xi_2\wih{\om}_1-\xi_1\wih{\om}_2)\right],\\
	%\enn
	%and
	%\ben
	\wih{v}_2
	%&=(\mu_js)^{-1}[-\frac{\xi_2}{\beta_j(\xi)}(\xi_1\wih{\om}_1+\xi_2\wih{\om}_2)
	%+\beta_j(\xi)\wih{\om}_2]\\
	&=\frac{1}{\mu_js\beta_j(\xi)}\left[\vep_j\mu_js^2\wih{\om}_2
	+\xi_1(\xi_1\wih{\om}_2-\xi_2\wih{\om}_1)\right],
\enn
where $\wih{v}_j$ denotes the Fourier transform of $v_j$ with respect to $\wit{x}$ (see Appendix \ref{ap1} for the definition of Fourier transform), and
\be\label{beta}
\beta_j(\xi)=(\vep_j\mu_js^2+|\xi|^2)^{1/2},\quad {\rm with}\;\Rt[\beta_j(\xi)]>0.
\en
 For convenience, we eliminate the magnetic field $\ch{\bm H}$ and get the TBCs for electric field $\ch{\bm E}$ in the s-domain and time domain, respectively:
\begin{align}\label{2.11}
    (\mu_js)^{-1}(\na\ti\ch{\bm{E}})\ti\n_j+\mathscr{B}_j[\ch{\bm{E}}_{\Ga_{h_j}}]=\0,\ \on\;\; \Ga_{h_j},
    \\ \label{2.13}
    (\mu_j)^{-1}(\na\ti\bm{E})\ti\n_j+\mathscr{C}_j[\bm{E}_{\Ga_{h_j}}]=\0,\ \on\;\; \Ga_{h_j},
\end{align}
where $\mathscr{C}_j=\mathscr{L}^{-1}\circ s\mathscr{B}_j\circ\mathscr{L}$.

The following lemma on the boundedness and coercivity of $\mathscr{B}_j$ plays a key role in the proof of the well-posedness which has been shown in \cite{GL2017}.

\begin{lem}\label{lem2.1}
    For $j=1,2$, $\mathscr{B}_j$ is continuous from $H^{-1/2}(\curl,\Ga_{h_j})$ to $H^{-1/2}(\dive,\Ga_{h_j})$
    (see Appendix \ref{ap2} for the definition of the trace spaces).
    Moreover, for any $\bm\om\in H^{-1/2}(\curl,\Ga_{h_j})$, we have
    \ben
    	\Rt\langle\mathscr{B}_j\bm\om,\bm\om\rangle_{\Ga_{h_j}}\geq 0.
    \enn
\end{lem}

% \section{The reduced problem}
% %\setcounter{equation}{0}

% In this section, we shall prove the well-posedness and stability of the reduced problem (\ref{2.7}) 
%and establish an a priori estimate in the end. For this purpose, we take two steps to finish our proof
% about well-posedness in time domain: we firstly obtain the well-posedness in s-domain, then apply 
% Lemma \ref{A.1} and energy method to deduce the main result. As for the prior estimate, we take 
% special test functions in the time-domain variational problem to achieve it.

\subsection{Well-posedness in s-domain}
Eliminating the magnetic field $\ch{\bm H}$ in (\ref{2.8}), we consider the reduced vector boundary value problem
\begin{subnumcases}{}\label{3.1a}
    \De^*\ch{\bm u}-\rho_is^2\ch{\bm u}=\0,& $\gin\;\; \Om$\\ \label{3.1b}
    \na\ti((s\mu)^{-1}\na\ti \ch{\bm{E}})+s\vep\ch{\bm{E}}=-\ch{\bm J},&$\gin\;\;  \Om_h$\\ \label{3.1c}
     \n\ti[\ch{\bm E}]=\n\ti[\mu^{-1}\na\ti\ch{\bm E}]=\0,&$\on\;\;\Ga_f$\\ \label{3.1c1}
    -(\mu_2s)^{-1}\n\ti\na\ti\ch{\bm{E}}=\bm{T\ch{u}},&$\on\;\; \Ga$\\ \label{3.1d}
    \n\ti\ch{\bm{E}}=\n\ti s\ch{\bm{u}},&$\on\;\; \Ga$\\ \label{3.1e}
    (\mu_js)^{-1}(\na\ti\ch{\bm{E}})\ti\n_j+\mathscr{B}_j[\ch{\bm{E}}_{\Ga_{h_j}}]=\0,&$\on\;\; \Ga_{h_j},\quad j=1,2,$
\end{subnumcases}
%We shall prove that Problem (\ref{3.1a})-(\ref{3.1f}) is well-posed
in the Hilbert space $\mathscr{X}_s:=\left\{(\bm V,\bm v)\in H(\curl,\Om_h)\ti H^1(\Om)^3,\;\n\ti\bm V=\n\ti s\bm v,\;\on\;\;\Ga\right\}$ under the norm
\be\label{3.1+}
	\V(\bm V,\bm{v})\V_{\mathscr{X}_s}:=\left(\V\bm V\V_{H(\curl,\Om_h)}^2+\V\bm v\V_{H^1(\Om)^3}^2\right)^{1/2}.
\en
We shall prove the well-posedness of problem (\ref{3.1a})-(\ref{3.1e}) in $\mathscr{X}_s$ by the Lax-Milgram lemma. To this end, we derive the variational formulation of (\ref{3.1a})-(\ref{3.1e}) by multiplying (\ref{3.1b}) and (\ref{3.1a}) with the complex conjugates of a pair of test functions ($\bm V,\bm{v})\in\mathscr{X}_s$, respectively, and applying integration by part, coupling interface condition (\ref{3.1c1}), and TBCs (\ref{3.1e}). Hence, the variational formulation of (\ref{3.1a})-(\ref{3.1e}) reads as follows: find a solution $(\ch{\bm E},\ch{\bm{u}})\in\mathscr{X}_s$ such that
\begin{align}\label{3.2}
      &\int_{\Om_h}(s\mu)^{-1}(\na\ti\ch{\bm{E}})\cdot(\na\ti\ov{\bm V})dx+\int_{\Om_h}s\vep\ch{\bm{E}}\cdot\ov{\bm V}dx\\ \no
      &+\int_{\Ga}(s\mu_2)^{-1}\na\ti\ch{\bm{E}}\ti\n\cdot\ov{\bm V}d\g+\sum_{j=1}^2\langle\mathscr{B}_j[\ch{\bm{E}}_{\Ga_{h_j}}],
        \bm{V}_{\Ga_{h_j}}\rangle_{\Ga_{h_j}}
      =-\int_{\Om_h}\ch{\bm J}\cdot\ov{\bm{V}}dx,
\end{align}
and
\begin{align}
    &\int_\Om\left[\ov{s}\Big(\la_i(\na\cdot\ch{\bm{u}})(\na\cdot\ov{\bm{v}})
	+2\mu_i\bm\vep(\ch{\bm{u}}):\bm\vep(\ov{\bm{v}})\Big)
    +\rho_is|s|^2\ch{\bm{u}}\cdot\ov{\bm v}\right]dx \no \\ \label{3.3}
    &-\int_{\Ga}(s\mu_2)^{-1}\na\ti\ch{\bm{E}}\ti\n\cdot\ov{s\bm v}d\g=0,
\end{align}
where $A:B=tr(AB^{\top})$ denotes the Frobenius inner product of square matrices A and B. Adding (\ref{3.3}) to (\ref{3.2}) gives the final variational form:
\be\label{3.4}
    a\left((\ch{\bm{E}},\ch{\bm{u}}),(\bm{V},\bm{v})\right)=-\int_{\Om_h}\ch{\bm J}\cdot\ov{\bm{V}}dx,
\en
where the sesquilinear form $a(\cdot,\cdot)$ is defined as
\begin{align}\label{3.5}
    a\left((\ch{\bm{E}},\ch{\bm{u}}),(\bm{V},\bm{v})\right)&=\int_{\Om_h}\big((s\mu)^{-1}(\na\ti\ch{\bm{E}})\cdot
    (\na\ti\ov{\bm V})dx+s\vep\ch{\bm{E}}\cdot\ov{\bm V}\big)dx \\ \no
    &+\int_\Om\Big[\ov{s}\mathcal{E}(\ch{\bm u},\ov{\bm v})+\rho_is|s|^2\ch{\bm{u}}\cdot\ov{\bm v}\Big]dx 
    +\sum_{j=1}^2\langle\mathscr{B}_j[\ch{\bm{E}}_{\Ga_{h_j}}],\bm{V}_{\Ga_{h_j}}\rangle_{\Ga_{h_j}}. 
\end{align}
Here, the bilinear form $\mathcal{E}(\bm u,\bm v)$ is defined by
\begin{align}\label{vep}
		\mathcal{E}(\bm u,\bm v) & := \la_i (\dive\;{\bm u}) (\dive\;{\bm v})
		+ 2 \mu_i {\bm \vep}({\bm u}):{\bm \vep}({\bm v}) \\ \no
		& = 2\mu_i\Big(\sum_{i,j=1}^{3}\pa_iu_j\pa_iv_j\Big)
		+\la_i(\dive\;\bm u)(\dive\;\bm v)-\mu_i\curl\;\bm u\cdot\curl\;\bm v.
	\end{align}
Under our assumptions on the Lam\'e constants: $\mu>0,3\la+2\mu>0$, we have the estimate (see \cite[Chap. 5.4]{Hsiao2008})
\be\label{estimate_vep}
	\int_{\Om} \mathcal{E}(\bm u,\ov{\bm u}) dx \ge C_\Om \|{\bm \vep}({\bm u})\|_{F(\Om)}^2,
\en
where the positive constant $C_\Om$ only depends on $\Om$, and $\|{\bm \vep}({\bm u})\|_{F(\Om)}$ denotes the Frobenius norm defined by
\ben
    \|{\bm \vep}({\bm u})\|_{F(\Om)}:=\Big(\sum_{i,j=1}^3\|\vep_{ij}(\bm u)\|_{L^2(\Om)}^2\Big)^{1/2}.
\enn
\begin{lem}\label{thm3.1}
    For each $s\in\Ct_+$, the variational problem \eqref{3.4} has a unique solution
     $(\ch{\bm{E}},\ch{\bm{u}})\in\mathscr{X}_s$ which
    satisfies the following estimates:
    \begin{align}
        &\|\na\ti\ch{\bm E}\|_{L^2(\Om_h)^3}+\| s\ch{\bm E}\|_{L^2(\Om_h)^3}\lesssim
        s_1^{-1}\|s\ch{\bm J}\|_{L^2(\Om_h)^3},\label{3.6}\\
        &\|\na\ch{\bm{u}}\|_{F(\Om)}+\| \na\cdot\ch{\bm{u}}\|_{L^2(\Om)}+\| s\ch{\bm{u}}\|_{{L^2(\Om)}^3}
        \lesssim s_1^{-1}\max\{1,s_1^{-1}\}\|\ch{\bm J}\|_{L^2(\Om_h)^3}.\label{3.7}
    \end{align}
    % where $\|\na\ch{{\bm u}}\|_{F(\Om)}$ denotes the Frobenius norm defined by
    % \ben
    %     \|\na\ch{{\bm u}}\|_{F(\Om)}:=\Big(\sum_{j=1}^3\int_{\Om}|\na\ch{{\bm u}}_j|^2dx\Big)^{1/2}.
    % \enn
   Hereafter, the expression $a\lesssim b$ or $a\gtrsim b$ stands for $a\le C b$ or $a\ge C b$, where $C$ is a positive constant and its specific value is not required but should be always clear from the context.
\end{lem}

\begin{proof}
    i) By Cauchy-Schwartz inequality, the boundness of $\mathscr{B}_j$ and Lemma \ref{lemB.2}, it follows that
    \ben
        \left|a\left((\ch{\bm{E}},\ch{\bm{u}}),(\bm V,\bm v)\right)\right|
        &\lesssim|s|^{-1}\|\na\ti\ch{\bm{E}}\|_{L^2(\Om_h)^3}\|\na\ti\bm V\|_{L^2(\Om_h)^3} \\
        &\quad+|s|\|\ch{\bm{E}}\|_{L^2(\Om_h)^3}\|\bm V\|_{L^2(\Om_h)^3}
        +|s|\|\na\cdot\ch{\bm{u}}\|_{L^2(\Om)}\|\na\cdot\bm v\|_{L^2(\Om)} \\
        &\quad+|s|^3\|\ch{\bm{u}}\|_{{L^2(\Om)}^3}\|\bm v\|_{{L^2(\Om)}^3}
        +|s|\|\na\ch{\bm{u}}\|_{F(\Om)}\|\na\bm v\|_{F(\Om)}  \\
        &\quad+\sum_{j=1}^2\|\mathscr{B}_j[\ch{\bm{E}}_{\Ga_{h_j}}]\|_{H^{-1/2}(\dive,\Ga_{h_j})}
        \|\bm V_{\Ga_{h_j}}\|_{H^{-1/2}(\curl,\Ga_{h_j})}  \\
        % +|s|^{-1}\|\g_t(\na\ti\ch{\bm{E}})\|_{H^{-1/2}(\Dive,\Ga)}  \no\\
        % &\quad\|\g_T\bm V\|_{H^{-1/2}(\Curl,\Ga)}
        % +\|\g_t(\na\ti\ch{\bm{E}})\|_{H^{-1/2}_{\parallel}(\Ga)}\|\g_T\bm v\|_{H^{1/2}_{\parallel}(\Ga)}  \no\\
        &\lesssim\|\ch{\bm{E}}\|_{H(\curl,\Om_h)}\|\bm V\|_{H(\curl,\Om_h)}+\|\ch{\bm{u}}\|_{H^1(\Om)^3}\|\bm v\|_{H^1(\Om)^3},
    \enn
    which yields that $a(\cdot,\cdot)$ is continuous in the product space $\mathscr{X}_s\ti\mathscr{X}_s$.

    ii) $a(\cdot,\cdot)$ is uniformly coercive. In fact, setting $(\bm{V},\bm{v}):=(\ch{\bm{E}},\ch{\bm{u}})$ in (\ref{3.5}) yields
    \begin{align}\label{3.9}
        a\left((\ch{\bm{E}},\ch{\bm{u}}),(\ch{\bm{E}},\ch{\bm{u}})\right)&=\int_{\Om_h}\big((s\mu)^{-1}|
        \na\ti\ch{\bm{E}}|^2dx+s\vep|\ch{\bm{E}}|^2\big)dx  \\\no
        &+\int_\Om\Big[\ov{s}\mathcal{E}(\ch{\bm u},\ov{\ch{\bm u}})+\rho_is|s\ch{\bm{u}}|^2\Big]dx
        +\sum_{j=1}^2\langle\mathscr{B}_j[\ch{\bm{E}}_{\Ga_{h_j}}],\ch{\bm{E}}_{\Ga_{h_j}}\rangle_{\Ga_{h_j}}. 
    \end{align}
    Define $\mu_{max}:=\max\{\mu_1,\mu_2\}, \vep_{min}:=\min\{\vep_1,\vep_2\}$.
    Combining the estimate \eqref{estimate_vep}
    and the well-known Korn's inequality \cite[Lemma 5.4.4]{Hsiao2008}
    \be\label{korn}
    \|\bm\vep(\bm v)\|^2_{F(\Om)}+\|\bm v\|^2_{L^2(\Om)^3}\geq C_{\Om}\|\bm v\|^2_{H^1(\Om)^3}, \qquad \forall\;\bm v\in H^1(\Om)^3
    \en
    then taking the real part of (\ref{3.9}) and using Lemma \ref{lem2.1}, we have
	\begin{align}
	    \Rt[a((\ch{\bm{E}},\ch{\bm{u}}),(\ch{\bm{E}},\ch{\bm{u}}))]
	    &\geq\frac{s_1}{|s|^2}(\mu_{max}^{-1}\|\na\ti\ch{ \bm E}\|_{L^2(\Om_h)^3}^2
	    +\vep_{min}\| s\ch{\bm E}\|_{L^2(\Om_h)^3}^2) \no\\
	    &\quad+s_1(C_{\Om}\|\bm\vep(\ch{\bm u})\|^2_{F(\Om)}
	    +\rho_i\|s\ch{\bm u}\|^2_{L^2(\Om)^3}) \no\\
	    &\geq \frac{s_1}{|s|^2}C_1\|\ch{\bm E}\|_{H(\curl,\Om_h)}^2+s_1C_2\|\ch{\bm u}\|_{H^1(\Om)^3}^2 \no\\
	    &\geq C\|(\ch{\bm{E}},\ch{\bm{u}})\|_{\mathscr{X}_s}^2, \label{coercivity}
	\end{align}
	where $C$ is defined as
	\ben
		C:=\min\{\frac{s_1}{|s|^2}C_1,s_1C_2\},C_1=\min\{\mu_{max}^{-1},\vep_{min}|s|^2\},
		C_2=C_{\Om}\min\{C_{\Om},\rho_i|s|^2\}.
	\enn

    It follows from the Lax-Milgram lemma that the variational problem (\ref{3.4}) has a unique solution
     $(\ch{\bm{E}},\ch{\bm{u}})\in\mathscr{X}_s$ for each $s\in\Ct_+$. Moreover, using (\ref{3.4}),
     we clearly have
    \begin{equation}\label{3.11}
        \begin{aligned}
            a\left((\ch{\bm{E}},\ch{\bm{u}}),(\ch{\bm{E}},\ch{\bm{u}})\right)
            &\lesssim\frac{1}{|s|}\|\ch{\bm J}\|_{{L^2}(\Om_h)^3}\| s\ch{\bm{E}}\|_{L^2(\Om_h)^3}\\
            &\leq\frac{1}{2\eps|s|^2}\|\ch{\bm J}\|_{L^2(\Om_h)^3}^2+\frac{\eps}{2}\| s\ch{\bm E}\|
            _{L^2(\Om_h)^3}^2,
        \end{aligned}
    \end{equation}
    where we have used $\eps$-inequality in the last inequality.

	Choosing $\eps$ sufficiently small such that $\frac{\eps}{2}<\frac{s_1}{|s|^2}$, e.g., $\eps=\frac{s_1}{|s|^2}$,
	combining (\ref{coercivity}) with (\ref{3.11}), we obtain
	\be\label{3.12}
	     &\quad\frac{s_1}{|s|^2}\left(\|\na\ti\ch{\bm E}\|_{L^2(\Om_h)^3}^2
	     +\| s\ch{\bm E}\|_{L^2(\Om_h)^3}^2\right)\\
	     &\quad +s_1\min\{1,s_1^2\}\left(\|\na\ch{\bm{u}}\|_{F(\Om)}^2+\|\na\cdot\ch{\bm{u}}\|_{{L^2(\Om)}}^2
	     +\| s\ch{\bm{u}}\|_{{L^2(\Om)}^3}^2\right)\\
	     &\lesssim s_1^{-1}\|\ch{\bm J}\|_{L^2(\Om_h)^3}^2,
	\en
	we arrive at (\ref{3.6}) and (\ref{3.7}) after using Cauchy-Schwartz inequality for (\ref{3.12}).
\end{proof}

\subsection{Well-posedness in time domain}
For $0\leq t\leq T$, to show the well-posedness of the reduced problem (\ref{2.7}) and
 the convergence of the PML method, we make the following assumptions on the source term $\bm J$:
\be\label{assumption}
    \bm J\in H^5(0,T;L^2(\Om_h)^3),\;\;\pa_t^l\bm J|_{t=0}=\bm 0, l=0,1,2,3,4.
\en
Furthermore, in the rest of the paper, we will always assume that $\bm J$ can be extended to $\ify$
with respect to $t$ such that
\be\label{assumption1}
    \bm J\in H^5(0,\ify;L^2(\Om_h)^3),\;\;\| \bm J\|_{H^5(0,\ify;L^2(\Om_h)^3)}
    \lesssim \| \bm J\|_{H^5(0,T;L^2(\Om_h)^3)}.
\en

\begin{theorem}\label{thm3.4}
    The reduced initial-boundary value problem \eqref{2.7} has a unique solution
     $\left(\bm E(x,t),\bm H(x,t),\bm{u}(x,t)\right)$ satisfying
    \ben
        &\bm E(x,t)\in L^2\left(0,T;H(\curl,\Om_h)\right)\cap H^1\left(0,T;L^2(\Om_h)^3\right),\\
        &\bm H(x,t)\in L^2\left(0,T;H(\curl,\Om_h)\right)\cap H^1\left(0,T;L^2(\Om_h)^3\right),\\
        &\bm{u}(x,t)\in L^2\left(0,T;H^1(\Om)^3\right)\cap H^1\left(0,T;L^2(\Om)^3\right),
    \enn
    with the stability estimate
    \begin{align}\label{estimate}
        &\max\limits_{t\in[0,T]}\big(\|\pa_t\bm E\|_{L^2(\Om_h)^3}+\|\na\ti\bm E\|_{L^2(\Om_h)^3}\\ \no
        &\quad\quad+\|\pa_t\bm H\|_{L^2(\Om_h)^3}+\|\na\ti\bm H\|_{L^2(\Om_h)^3}\big)
        \lesssim{}\|\bm J\|_{H^1(0,T;L^2 (\Om_h)^3)},
     \end{align}
     \be\label{estimate1}
       &\max\limits_{t\in[0,T]}\big(\|\pa_t\bm{u}\|_{L^2(\Om)^3}+\|\na\cdot\bm{u}\|_{L^2(\Om)}
       +\|\na\bm{u}\|_{F(\Om)}\big)
        \lesssim{}\|\bm J\|_{L^1(0,T;L^2 (\Om_h)^3)}.
    \en
\end{theorem}

\begin{proof}
    Simple calculations yields the following estimate
    \ben
        &\int_0^T(\|\na\ti\bm E\|_{L^2(\Om_h)^3}^2+\|\pa_t\bm E\|_{L^2(\Om_h)^3}^2
        +\|\na\bm{u}\|_{F(\Om)}^2+\|\pa_t\bm{u}\|_{L^2(\Om)^3}^2)dt\\
        \leq{}&\int_0^Te^{-2s_1(t-T)}(\|\na\ti\bm E\|_{L^2(\Om_h)^3}^2
        +\|\pa_t\bm E\|_{L^2(\Om_h)^3}^2+\|\na\bm{u}\|_{F(\Om)}^2
        +\|\pa_t\bm{u}\|_{L^2(\Om)^3}^2)dt\\
        \lesssim{}&\int_0^{\ify}e^{-2s_1t}(\|\na\ti\bm E\|_{L^2(\Om_h)^3}^2
        +\|\pa_t\bm E\|_{L^2(\Om_h)^3}^2+\|\na\bm{u}\|_{F(\Om)}^2
        +\|\pa_t\bm{u}\|_{L^2(\Om)^3}^2)dt.
    \enn
  It is therefore sufficient to estimate the integral
  \ben
	  \int_0^{\ify}e^{-2s_1t}(\|\na\ti\bm E\|_{L^2(\Om_h)^3}^2+\|\pa_t\bm E\|_{L^2(\Om_h)^3}^2
	  +\|\na\bm{u}\|_{F(\Om)}^2+\|\pa_t\bm{u}\|_{L^2(\Om)^3}^2)dt.
  \enn

    Recalling the $s$-domain reduced system  (\ref{2.8}), by estimates (\ref{3.6}) and (\ref{3.7})
    in Lemma \ref{thm3.1}, it follows from \cite[Lemma 44.1]{treves1975} that $(\ch{\bm{E}},\ch{\bm{u}})$
     are holomorphic functions of $s$ on the half plane $s_1>\g>0,$ where $\g$ is any positive constant.
     Hence we have from Lemma \ref{lemA.2} that the inverse Laplace transform of $\ch{\bm{E}}$ and
     $\ch{\bm{u}}$ exist and are supported in $[0,\ify]$.

    Denote by $\bm E=\mathscr{L}^{-1}(\ch{\bm E})$ and $\bm{u}=\mathscr{L}^{-1}(\ch{\bm{u}})$.
    It follows that using the Parseval identity (\ref{A.5}) and estimate (\ref{3.6})
	\ben
	     &\quad\int_0^{\ify}e^{-2s_1t}(\|\na\ti\bm E\|_{L^2(\Om_h)^3}^2+\|\pa_t\bm E\|_{L^2(\Om_h)^3}^2)dt\\
	     &=\frac{1}{2\pi}\int_{-\ify}^{\ify}(\|\na\ti\ch{\bm{E}}\|^2_{L^2(\Om_h)^3}
	     +\| s\ch{\bm{E}}\|^2_{L^2(\Om_h)^3})ds_2\\
	     &\lesssim\int_{-\ify}^{\ify}s^{-2}_1\|s\ch{\bm J}\|_{L^2(\Om_h)^3}^2ds_2
	     % ={}&s_1^{-2}\int_{-\ify}^{\ify}\| \mathscr{L}(\pa_t\bm J)\|_{L^2(\Om_h)^3}^2ds_2\\
	     \lesssim s_1^{-2}\int_0^{\ify}e^{-2s_1t}\|\pa_t\bm J\|_{L^2(\Om_h)^3}^2dt,
	\enn
    which shows that
    \ben
    	\bm E(x,t)\in L^2\left(0,T;H(\curl,\Om_h)\right)\cap H^1\left(0,T;L^2(\Om_h)^3\right),
    \enn
    thanks to the Maxwell system in (\ref{2.7}), we also have
    \ben
    	\bm H(x,t)\in L^2\left(0,T;H(\curl,\Om_h)\right)\cap H^1\left(0,T;L^2(\Om_h)^3\right).
    \enn
    For elastic wave, combining Parseval identity (\ref{A.5}) with estimate (\ref{3.7}), we similarly  have
	\ben
	    \int_0^{\ify}e^{-2s_1t}(\|\pa_t\bm{u}\|^2_{L^2(\Om)^3}+\|\na\bm{u}\|^2_{F(\Om)})dt
	    ={}&\frac{1}{2\pi}\int_{-\ify}^{\ify}(\| s\ch{\bm{u}}\|^2_{{L^2(\Om)}^3}+\|\na\ch{\bm{u}}\|^2_{F(\Om)})ds_2\\
	   \lesssim{}&\int_{-\ify}^{\ify}s^{-2}_1\max\{1,s_1^{-2}\}\|s\ch{\bm J}\|_{L^2(\Om_h)^3}^2ds_2\\
	    % ={}&s_1^{-2}\int_{-\ify}^{\ify}\| \mathscr{L}(\pa_t\bm J)\|_{L^2(\Om_h)^3}^2ds_2\\
	    \lesssim{}&s_1^{-2}\max\{1,s_1^{-2}\}\int_0^{\ify}e^{-2s_1t}\|\pa_t\bm J\|_{L^2(\Om_h)^3}^2dt,
	\enn
    which means that
    \ben
    	\bm{u}(x,t)\in L^2\left(0,T;H^1(\Om)^3\right)\cap H^1\left(0,T;L^2(\Om)^3\right).
    \enn

   In what follows, we shall prove the stability of solution in (\ref{2.7}) by means of the initial conditions.
   We start by defining an energy function
   \ben
   	\vep(t)=e_1(t)+e_2(t),\quad \for\;\;t\in(0,T)
   \enn
   with
    \ben
        &e_1(t)=\|\vep^{1/2}\bm E(\cdot,t)\|_{L^2(\Om_h)^3}^2+\|\mu^{1/2}\bm H(\cdot,t)\|_{L^2(\Om_h)^3}^2,\\
        &e_2(t)=\V\rho_{i}^{1/2}\pa_t\bm{u}\V_{L^2(\Om)^3}^2
        +\int_{\Om}\mathcal{E}(\bm u,\ov{\bm u})dx.
    \enn
    Observe that $\vep(\cdot)$ can be equivalently written as
    \be\label{3.14}
        \vep(t)-\vep(0)=\int_0^t\vep^{\prime}(\tau)d\tau=\int_0^t\Big(e_1^{\prime}(\tau)+e_2^{\prime}(\tau)\Big)d\tau.
    \en
    By simple calculations using the system (\ref{2.7}) and integration by parts, we have
    \be\label{3.15}
        \int_0^te_1^{\prime}(\tau)d\tau&=2\Rt\int_0^t\int_{\Om_h}\left(\vep\pa_\tau\bm E\cdot\ov{\bm E}
        +\mu\pa_\tau\bm H\cdot\ov{\bm H}\right)dxd\tau\\
%        &=2\Rt\int_0^t\int_{\Om_h}\left((\na\ti \bm{H})\cdot\ov{\bm E}-(\na\ti \bm{E})\cdot\ov{\bm H}\right)
%        dxd\tau-2\Rt\int_0^t\int_{\Om_h}\bm J\cdot\ov{\bm E}dxd\tau\\
        &=2\Rt\int_0^t\int_{\Om_h}\left((\na\ti \ov{\bm{E}})\cdot\bm H-(\na\ti \bm{E})\cdot\ov{\bm H}\right)dxd\tau\\
        &\quad-2\Rt\sum_{j=1}^2\int_0^t\int_{\Ga_{h_j}}\mathscr{T}_j[\bm E_{\Ga_{h_j}}]\cdot\ov{\bm E}_{\Ga_{h_j}}d\g d\tau\\
        &\quad-2\Rt\int_0^t\int_{\Om_h}\bm J\cdot\ov{\bm E}dxd\tau+2\Rt\int_0^t\int_{\Ga}(\bm{H}\ti\n)\cdot\ov{\bm E}d\g d\tau\\
        &=-2\Rt\sum_{j=1}^2\int_0^t\int_{\Ga_{h_j}}\mathscr{T}_j[\bm E_{\Ga_{h_j}}]\cdot\ov{\bm E}_{\Ga_{h_j}}d\g d\tau\\
        &\quad-2\Rt\int_0^t\int_{\Om_h}\bm J\cdot\ov{\bm E}dxd\tau+2\Rt\int_0^t\int_{\Ga}(\bm{H}\ti\n)\cdot\ov{\bm E}d\g d\tau.
    \en
    Noting the definition of $\mathcal{E}(\bm u,\bm v)$ (see \eqref{vep}), by the elastic wave equation \eqref{2.1} 
    and using the integration by parts, it can be similarly shown that 
    \begin{align}\label{3.16}
	    \int_0^te_2^{\prime}(\tau)d\tau
	    &= 2\Rt\int_0^t\int_{\Om}\Big(\rho_e\pa_\tau^2\bm{u}\cdot\pa_\tau\ov{\bm{u}}
	    + \mathcal{E}(\pa_{\tau}\bm u,\ov{\bm u})\Big) dx d\tau \no\\
	    % &= 2\Rt\int_0^t\int_{\Om_1}\Big((\De^{*}\bm{u}+{\bm j})\cdot\pa_\tau\ov{\bm{u}}
	    % + \mathcal{E}(\pa_{\tau}\bm u,\ov{\bm u})\Big) dx d\tau \no\\
	    &= \int_0^t\int_{\Om}2\Rt\Big(- \mathcal{E}(\bm u,\pa_{\tau}\ov{\bm u}) 
	    + \mathcal{E}(\pa_{\tau}\bm u,\ov{\bm u})\Big)dx d\tau \no \\
	    &\quad 
	    + 2\Rt\int_0^t\int_{\Ga}\bm{Tu}\cdot\pa_\tau\ov{\bm{u}}d\g d\tau
	    =2\Rt\int_0^t\int_{\Ga}\bm{Tu}\cdot\pa_\tau\ov{\bm{u}}d\g d\tau.
    \end{align}
    Combining (\ref{3.14})-(\ref{3.16}) with $\vep(0)=0$  and the interface condition (\ref{2.4}), we obtain
    \be\label{3.17}
        \vep(t)&=-2\Rt\sum_{j=1}^2\int_0^t\int_{\Ga_{h_j}}\mathscr{T}_j[\bm E_{\Ga_{h_j}}]\cdot\ov{\bm E}_{\Ga_{h_j}}d\g d\tau
        -2\Rt\int_0^t\int_{\Om_h}\bm J\cdot\ov{\bm E}dxd\tau\\
        &\quad+2\Rt\int_0^t\int_{\Ga}(\bm{H}\ti\n)\cdot\ov{\bm E}d\g d\tau
        +2\Rt\int_0^t\int_{\Ga}\bm{Tu}\cdot\pa_\tau\ov{\bm{u}}d\g d\tau\\
        &=-2\Rt\sum_{j=1}^2\int_0^t\int_{\Ga_{h_j}}\mathscr{T}_j[\bm E_{\Ga_{h_j}}]\cdot\ov{\bm E}_{\Ga_{h_j}}d\g d\tau
        -2\Rt\int_0^t\int_{\Om_h}\bm J\cdot\ov{\bm E}dxd\tau\\
        &\quad\;+2\Rt\int_0^t\int_{\Ga}\bm H\cdot(\n\ti\ov{\bm E}-\n\ti\pa_{\tau}\ov{\bm u})d\g d\tau\\
        &=-2\Rt\sum_{j=1}^2\int_0^t\int_{\Ga_{h_j}}\mathscr{T}_j[\bm E_{\Ga_{h_j}}]\cdot\ov{\bm E}_{\Ga_{h_j}}d\g d\tau
        -2\Rt\int_0^t\int_{\Om_h}\bm J\cdot\ov{\bm E}dxd\tau.
    \en
  By \cite[equation (4.11)]{GL2017}, it holds that
  \ben
  	\Rt\int_0^t\int_{\Ga_{h_j}}\mathscr{T}_j[\bm E_{\Ga_{h_j}}]\cdot\ov{\bm E}_{\Ga_{h_j}}d\g d\tau \ge 0.
  \enn
 This, combining the $\eps$-inequality and \eqref{estimate_vep} one has the following estimate
  \be\label{3.18}
	&\|\bm E(\cdot,t)\|_{L^2(\Om_h)^3}^2+\|\bm H(\cdot,t)\|_{L^2(\Om_h)^3}^2
	+\|\pa_t\bm{u}\|^2_{L^2(\Om)^3}+\|\bm\vep(\bm{u})\|^2_{F(\Om)}\\
	\lesssim{}&\vep(t)\leq{}-2\Rt\int_0^t\int_{\Om_h}\bm J\cdot\ov{\bm E}dxd\tau
	% \lesssim{}&\int_0^t\|\bm\rho\|_{H^{-1/2}(\dive,\Ga_{h_1})}\cdot\|\bm E_{\Ga_{h_1}}\|_{H^{-1/2}(\curl,\Ga_{h_1})}d\tau\no\\
	\lesssim{}\int_0^t\|\bm J\|_{L^2(\Om_h)^3}\cdot\|\bm E\|_{L^2(\Om_h)^3}d\tau\\
	\lesssim{}&\frac{\eps}{2}\max\limits_{t\in[0,T]}\|\bm E(\cdot,t)\|_{L^2(\Om_h)^3}^2+\frac{1}{2\eps}\|\bm J\|^2_{L^1(0,T;L^2(\Om_h)^3)}.
   \en
    Finally, letting $\eps>0$ in (\ref{3.18}) small enough, e.g. $\eps=1$ and applying Cauchy-Schwartz inequality yields
	\be\label{3.18'}
	   &\max\limits_{t\in[0,T]}\Big(\|\bm E(\cdot,t)\|_{L^2(\Om_h)^3}+\|\bm H(\cdot,t)\|_{L^2(\Om_h)^3}\\
	   &\qquad\quad+\|\pa_t\bm{u}\|_{L^2(\Om)^3}+\|\bm{\vep}(\bm{u})\|_{F(\Om)}\Big)\\
	   &\lesssim\|\bm J\|_{L^1(0,T;L^2(\Om_h)^3)}\lesssim\|\bm J\|_{L^2(0,T;L^2(\Om_h)^3)}.
	\en
	Now, by using the Cauchy-Schwartz inequality again, we have for any $0\le\xi \le T$
	\begin{align}\label{3.18+}
	\|\bm{u}(\cdot,\xi)\|_{L^2(\Om)^3}^2&=\int_0^\xi\pa_t\|\bm{u}(\cdot,t)\|_{L^2(\Om)^3}^2dt
	=2\Rt\int_0^\xi\int_\Om\pa_t\bm u(x,t)\cdot\bm u(x,t)dxdt\notag\\
	&\le 2\int_0^\xi\Big(\eps\|\pa_t\bm u(\cdot,t)\|^2_{L^2(\Om)^3}
	 +\frac{1}{4\eps}\|\bm u(\cdot,t)\|^2_{L^2(\Om)^3}\Big)dt\notag\\
	&\lesssim 2T\eps\|\pa_t\bm u(\cdot,\xi)\|^2_{L^2(\Om)^3}
	 +\frac{T}{2\eps}\|\bm u(\cdot,\xi)\|^2_{L^2(\Om)^3}.
	\end{align}
	Choosing $\eps=T$ in (\ref{3.18+}) gives
	\be\label{3.19}
		\|\bm{u}(\cdot,\xi)\|_{L^2(\Om)^3}^2\lesssim T^2\|\partial_t\bm u(\cdot,\xi)\|^2_{L^2(\Om)^3}.
	\en
	Applying Korn's inequality \eqref{korn} and using (\ref{3.19}) gives
	 \ben
		 \V\pa_t\bm{u}\V^2_{L^2(\Om)^3}+\V\bm\vep(\bm{u})\V^2_{F(\Om)}\gtrsim\|\bm u\|^2_{H^1(\Om)^3}
		 \gtrsim\|\na\cdot\bm u\|^2_{L^2(\Om)}+\|\na\bm u\|^2_{F(\Om)},
	 \enn
	This, combining (\ref{3.18'}) leads to the stability estimate (\ref{estimate1}).
	
	Taking the derivative of (\ref{2.7}) with respect to $t$, observing that $(\pa_t\bm E,\pa_t\bm H)$
	satisfy the same set of equations with the source $\bm J$ replaced by $\pa_t\bm J$, and
	 the initial conditions replaced by $\pa_t\bm E=\pa_t\bm H=\0$ using (\ref{2.2c})-(\ref{2.2d})
	 and $\pa_t\bm u$ also satisfies elastodynamic equation with $\pa_t^2\bm u(x,0)=\rho_i^{-1}
	 \De^*\bm u(x,0)=\0$, therefore we can follow the same steps as deriving (\ref{3.18'}) for
	 $(\pa_t\bm E,\pa_t\bm H)$  which leads to
	\be\label{3.18''}
	   &\max\limits_{t\in[0,T]}\Big(\V\pa_t\bm E(\cdot,t)\V_{L^2(\Om_h)^3}+\V\pa_t\bm H(\cdot,t)\V_{L^2(\Om_h)^3}\\
	   &{}+\V\pa_t^2\bm{u}\V_{L^2(\Om)^3}+\V\bm\vep(\pa_t\bm{u})\V_{F(\Om)}\Big)
	   \les\V\pa_t\bm J\V_{L^2(0,T;L^2(\Om_h)^3)}.
	\en
	This, combining (\ref{3.18'}) with the Maxwell's equations completes our proof of (\ref{estimate}).
\end{proof}

\section{The time domain PML problem}\label{sec-pml}
In this section, we shall derive the time domain PML formulation of the electromagnetic-elastic
interaction scattering problem. The well-posedness and stability of the PML problem is established 
based on the variational method and the energy method which is adopt in section \ref{sec-well_posed}.
In the end, we shall show the exponential convergence analysis of the time domain PML method applying
a novel technique to construct the PML layer.
\begin{figure}[!htbp]
\setcounter{subfigure}{0}
  \centering
  \includegraphics[width=4in]{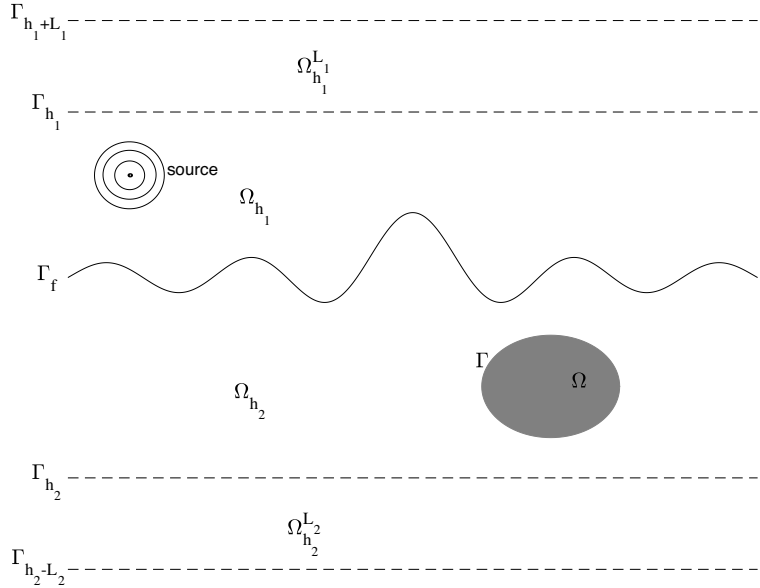}
 \caption{Geometric configuration of the truncated PML problem}\label{PML}
\end{figure}

\subsection{The PML equations and Well-posedness}
We firstly introduce geometry of the PML problem as shown in Figure \ref{PML}.
Let $\Om_{h_1}^{L_1}:=\{x\in\R^3: h_1<x_3<h_1+L_1\}$ and $\Om_{h_2}^{L_2}:=\{x\in\R^3: h_2-L_2<x_3<h_2\}$
denote the PML layers with thickness $L_1$ and $L_2$ which surround the strip domain $\Om_h$.
Denote by $\Om_{h+L}:=\{x\in\R^3: h_2-L_2<x_3<h_1+L_1\} \;\cap\; \Om^c$ the truncated PML
domain with boundaries $\Ga_{h_1+L_1}:=\{x\in\R^3: x_3=h_1+L_1\}$ and $\Ga_{h_2-L_2}:=\{x\in\R^3: x_3=h_2-L_2\}$.
Now, let $s_1> 0$ be an arbitrarily fixed parameter and let us introduce the PML medium property $\sig=\sig(x_3)$:
\be\label{4.1}
    \sig(x_3)=
            \begin{cases}
                1,&\quad \text{if}\ h_2\leq x_3\leq h_1,\\
                1+s_1^{-1}\sig_1(\frac{x_3-h_1}{L_1})^m,&\quad \text{if}\;h_1<x_3<h_1+L_1,\\
                1+s_1^{-1}\sig_2(\frac{h_2-x_3}{L_2})^m,&\quad \text{if}\;h_2-L_2<x_3<h_2,
            \end{cases}
\en
where $\sig_j$ are two positive constants and $m\geq 1$ denotes a given integer.
In what follows, we will take the real part of the Laplace transform variable $s\in\Ct_+$ to be
$s_1$, that is, $\Rt(s)=s_1$.

Next, we shall derive the PML equations by the change of variables technique, starting by introducing
 the real stretched coordinate $\hat{x}$
\ben
    \hat{x}_1=x_1, \hat{x}_2=x_2, \hat{x}_3=\int_{0}^{x_3}\sig(\tau)d\tau.
\enn
Since ${\rm supp}\;\bm J\subset\Om_h$,
taking the Laplace transform of the original Maxwell's equation (\ref{2.2}) with respect to $t$,
we have for $j=1,2$
\be\label{4.2}
    \begin{cases}
        \na\ti \ch{\bm{E}}+\mu_j s \ch{\bm{H}}=\0,&\gin\;\;\Om_{h_j}^{L
        _j}\\
        \na\ti \ch{\bm{H}}-\vep_j s \ch{\bm{E}}=\0,&\gin\;\;\Om_{h_j}^{L
        _j}
    \end{cases}
\en
Let $\ch{\bm E}(\hat{x})$ and $\ch{\bm H}(\hat{x})$ be the PML extensions of the electromagnetic field
 $\ch{\bm E}$ and $\ch{\bm H}$ satisfying (\ref{4.2}). To be more precise, the change of variables 
 technique is to require $\ch{\bm E}(\hat{x})$ and $\ch{\bm H}(\hat{x})$ satisfying
\begin{equation}\label{4.3}
    \begin{cases}
        \na_p\ti \ch{\bm{E}}(\hat{x})+\mu_j s \ch{\bm{H}}(\hat{x})=\0,&\gin\;\;\Om_{h_j}^{L
        _j}\\
        \na_p\ti \ch{\bm{H}}(\hat{x})-\vep_j s \ch{\bm{E}}(\hat{x})=\0,&\gin\;\;\Om_{h_j}^{L
        _j}
    \end{cases}
\end{equation}
where $\na_p\ti\bm u:=(\pa_{x_2}u_3-\sig^{-1}\pa_{x_3}u_2,\sig^{-1}\pa_{x_3}u_1-\pa_{x_1}u_3,
\pa_{x_1}u_2-\pa_{x_2}u_1)^{\top}$ for any vector $\bm u=(u_1,u_2,u_3)^{\top}$.
Observing that
\ben
	\na\ti \diag(1,1,\sig)\bm u=\diag(\sig,\sig,1)\na_p\ti\bm u,
\enn
we introduce the PML solutions $(\ch{\wit{\bm E}},\ch{\wit{\bm H}})$ by
\begin{align}\label{4.4}
\ch{\wit{\bm E}}(x)&=\diag(1,1,\sig)\ch{\bm E}(\hat{x}),\\
\label{4.5}
\ch{\wit{\bm H}}(x)&=\diag(1,1,\sig)\ch{\bm H}(\hat{x}).
\end{align}
Inserting (\ref{4.4}) and (\ref{4.5}) into (\ref{4.3}) and combining the elastic wave equations, 
we obtain the truncated  PML equations of $\ch{\wit{\bm E}}$, $\ch{\wit{\bm H}}$ and $\ch{\wit{\bm u}}$
\begin{equation}\label{4.6}
    \begin{cases}
    \na\ti\ch{\wit{\bm E}}+\wit{\mu}s\ch{\wit{\bm H}}=\0,&\gin\;\;\Om_{h+L}\\
    \na\ti\ch{\wit{\bm H}}-\wit{\vep}s\ch{\wit{\bm E}}=\ch{\bm J},&\gin\;\;\Om_{h+L}\\
    \De^*\ch{\wit{\bm u}}-\rho_is^2\ch{\wit{\bm u}}=\0,& \gin\;\; \Om\\
    \n\ti[\ch{\wit{\bm E}}]=\n\ti[\mu^{-1}\na\ti\ch{\wit{\bm E}}]=\0,&\on\;\;\Ga_f\\
    % \n\ti[\ch{\wit{\bm H}}]=\n\ti[\vep^{-1}\na\ti\ch{\wit{\bm H}}]=\0, &\on\;\;\Ga_f\\
    \n\ti\ch{\wit{\bm H}}=\bm{T\ch{\wit{\bm u}}},\;\n\ti\ch{\wit{\bm E}}=\n\ti s\ch{\wit{\bm u}}&\on\;\; \Ga\\
    % \n\ti\ch{\wit{\bm E}}=\n\ti s\ch{\wit{\bm u}},&\on\;\; \Ga\\
    \ch{\wit{\bm E}}\ti\n_j=\0,&\on\;\;\Ga_{h_j\pm L_j},\quad j=1,2,
    % \ch{\wit{\bm E}}\ti\n_2=\0,&\on\;\;\Ga_{h_2-L_2}
    \end{cases}
\end{equation}
where $\wit{\mu}:=\diag(\sig,\sig,\sig^{-1})\mu$ and $\wit{\vep}:=\diag(\sig,\sig,\sig^{-1})\vep$,
respectively, and the perfect electric conductor (PEC) boundary conditions have been
imposed on the PML boundary $\Ga_{h_1+L_1}$ and $\Ga_{h_2-L_2}$ (Hereafter, we always take the 
sign $+$ when $j=1$, and $-$ when $j=2$ in $\Ga_{h_j\pm L_j}$).

Eliminating the magnetic field $\ch{\wit{\bm H}}$ from (\ref{4.6}) yields the equation 
of $(\ch{\wit{\bm E}},\ch{\wit{\bm u}})$
\begin{equation}\label{4.7}
    \begin{cases}
     \na\ti((s\wit{\mu})^{-1}\na\ti\ch{\wit{\bm E}})+s\wit{\vep}\ch{\wit{\bm E}}=-\ch{\bm J},&\gin\;\;\Om_{h+L}\\
     \De^*\ch{\wit{\bm u}}-\rho_is^2\ch{\wit{\bm u}}=\0,& \gin\;\; \Om\\
      \n\ti[\ch{\wit{\bm E}}]=\n\ti[\mu^{-1}\na\ti\ch{\wit{\bm E}}]=\0,&\on\;\;\Ga_f\\
    -\n\ti(s\wit{\mu})^{-1}\na\ti\ch{\wit{\bm E}}=\bm{T\ch{\wit{\bm u}}},&\on\;\; \Ga\\
    \n\ti\ch{\wit{\bm E}}=\n\ti s\ch{\wit{\bm u}},&\on\;\; \Ga\\
    \ch{\wit{\bm E}}\ti\n_j=\0,&\on\;\;\Ga_{h_j\pm L_j},\quad j=1,2.
    % \ch{\wit{\bm E}}\ti\n_2=\0,&\on\;\;\Ga_{h_2-L_2}
    \end{cases}
\end{equation}
In the following, we shall show the well-posedness of (\ref{4.7}) by the variational method in the Hilbert space
\ben
	\wit{\mathscr{X}}_s:=\left\{(\bm V,\bm v)\in H_0(\curl,\Om_{h+L})\ti H^1(\Om)^3,
	\;\n\ti\bm V=\n\ti s\bm v,\;\on\;\;\Ga\right\}
\enn
  where $H_0(\curl,\Om_{h+L}):=\{\bm u\in H(\curl,\Om_{h+L}):\bm u\ti\n_1=\0,\;\on
 \;\Ga_{h_1+L_1}\;\gand\;\bm u\ti\n_2=\0,\;\on\;\Ga_{h_2-L_2}\}$.
 And the norm on $\wit{\mathscr{X}}_s$ is defined as (\ref{3.1+}) with $\Om_h$ replaced by $\Om_{h+L}$.
 To this end, we introduce the variational formulation of (\ref{4.7}):
 to find a solution $(\ch{\wit{\bm E}},\ch{\wit{\bm u}})\in \wit{\mathscr{X}}_s$ such that
\be\label{4.8}
    \wit{a}\big((\ch{\wit{\bm E}},\ch{\wit{\bm u}}),(\bm{V},\bm v)\big)
    =-\int_{\Om_{h}}\ch{\bm J}\cdot\ov{\bm V}dx
    \quad \text{for all}\;\;(\bm V,\bm v)\in \wit{\mathscr{X}}_s,
\en
where the sesquilinear form $\wit{a}(\cdot,\cdot)$ is defined as
\ben
    \wit{a}\big((\ch{\wit{\bm E}},\ch{\wit{\bm u}}),(\bm{V},\bm v)\big)&=\int_{\Om_{h+L}}
    \Big((s\wit{\mu})^{-1}(\na\ti\ch{\wit{\bm{E}}})\cdot(\na\ti\ov{\bm V})dx
    +s\wit{\vep}\ch{\wit{\bm{E}}}\cdot\ov{\bm V}\Big)dx\\
    &+\int_\Om\left[\ov{s}\mathcal{E}(\ch{\wit{\bm{u}}},\ov{\bm v})
    +\rho_is|s|^2\ch{\wit{\bm{u}}}\cdot\ov{\bm v}\right]dx.
\enn
Noting that $1\leq\sig\leq1+s_1^{-1}\sig_0$, for $x\in\Om_{h+L}$,
combining the boundness of $\vep$, $\mu$, Korn's inequality \eqref{korn} and 
\eqref{estimate_vep}, we have
\ben
    \Rt\;\wit{a}\big((\ch{\wit{\bm E}},\ch{\wit{\bm u}}),(\ch{\wit{\bm E}},\ch{\wit{\bm u}})\big)
    \gtrsim{}&\frac{1}{1+s_1^{-1}\sig_0}\frac{s_1}{|s|^2}\left(\|\na\ti\ch{\wit{\bm E}}\|_{L^2(\Om_{h+L})^3}^2
    +\| s\ch{\wit{\bm E}}\|_{L^2(\Om_{h+L})^3}^2\right)\\
     &+s_1\min\{1,s_1^2\}\left(\|\na\ch{\wit{\bm{u}}}\|_{F(\Om)}^2+\|\na\cdot\ch{\wit{\bm{u}}}\|_{{L^2(\Om)}}^2
     +\| s\ch{\wit{\bm{u}}}\|_{{L^2(\Om)}^3}^2\right).
\enn
where $\sig_0:=\max\{\sig_1,\sig_2\}$, which implies the uniform coercivity of $\wit{a}(\cdot,\cdot)$.

Arguing similarly as in the proof of Lemma \ref{thm3.1} (noting that the TBC in the s-domain is now replaced 
with the PEC boundary condition), we can obtain the following lemma.
\begin{lem}\label{thm4.1}
  The truncated PML variational problem \eqref{4.8} has a unique solution $(\ch{\wit{\bm E}},\ch{\wit{\bm u}})
  \in\wit{\mathscr{X}}_s$ for each $s\in\Ct_+$ with $\Rt(s)=s_1>0$. Further, it holds that
\begin{align}\label{4.9}
    &\|\na\ti\ch{\wit{\bm E}}\|_{L^2(\Om_{h+L})^3}
    +\| s\ch{\wit{\bm E}}\|_{L^2(\Om_{h+L})^3}\lesssim s_1^{-1}(1+s_1^{-1}\sig_0)\| s\ch{\bm J}\|_{L^2(\Om_h)^3},
    \\ \label{4.9a}
    &\|\na\ch{\wit{\bm{u}}}\|_{F(\Om)}+\| \na\cdot\ch{\wit{\bm{u}}}\|_{L^2(\Om)}+\| s\ch{\wit{\bm{u}}}\|_
    {{L^2(\Om)}^3}\lesssim \frac{\sqrt{1+s_1^{-1}\sig_0}}{s_1\min\{1,s_1\}}\|\ch{\bm J}\|_{L^2(\Om_h)^3}.
\end{align}
\end{lem}

Taking the inverse Laplace transform of system (\ref{4.6}),
we obtain the truncated PML problem in the time domain
\begin{equation}\label{4.10}
    \begin{cases}
    \na\ti\wit{\bm E}+\wit{\mu}\pa_t\wit{\bm H}=\0,&\gin\;\;\Om_{h+L}\ti(0,T)\\
    \na\ti\wit{\bm H}-\wit{\vep}\pa_t\wit{\bm E}=\bm J,&\gin\;\;\Om_{h+L}\ti(0,T)\\
     \rho_i\frac{\pa^2\wit{\bm u}}{\pa t^2}-\De^*\wit{\bm u}=\0,& \gin\;\; \Om\ti(0,T)\\
    \wit{\bm E}|_{t=0}=\wit{\bm H}|_{t=0}=\0,&\gin\;\;\Om_{h+L}\\
     \wit{ \bm u}(x,0)=\pa_t\wit{\bm u}(x,0)=\0,& \gin\;\;  \Om\\
     \n\ti[\wit{\bm E}]=\n\ti[\mu^{-1}\na\ti\wit{\bm E}]=\0,&\on\;\;\Ga_f \ti(0,T)\\
    % \n\ti[\wit{\bm H}]=\n\ti[\vep^{-1}\na\ti\wit{\bm H}]=\0, &\on\;\;\Ga_f \ti(0,T)\\
    \wit{\bm E}\ti\n_j=\0,&\on\;\;\Ga_{h_j\pm L_j}\ti(0,T),\quad j=1,2.\\
%    \wit{\bm E}\ti\n_2=\0.&\on\;\;\Ga_{h_2-L_2}\ti(0,T)
    \end{cases}
\end{equation}
Note that $s_1$ appearing in the matrix $\wit{\mu}$ and $\wit{\vep}$ is an arbitrarily fixed, 
positive parameter, as mentioned earlier at the beginning of this subsection. In the Laplace transform domain,
the transform variable $s\in\Ct_+$ is taken so that $\Rt(s)=s_1>0$, and in the subsequent study of
the well-posedness and convergence of the truncated PML problem (\ref{4.10}), we take $s_1=1/T$.

The well-posedness and stability of the truncated PML problem in the time domain (\ref{4.10})
can be obtained similarly as Theorem \ref{thm3.4} 
with using the estimate (\ref{4.9})-(\ref{4.9a}) in Lemma \ref{thm4.1} as well as the energy method.

\begin{theorem}
Let $s_1=1/T$. The truncated initial-boundary value problem \eqref{4.10} has a unique solution 
$\big(\wit{\bm E}(x,t),\wit{\bm H}(x,t),\wit{\bm u}(x,t)\big)$ satisfying
\ben
    &\wit{\bm E}(x,t)\in L^2\big(0,T;H_0(\curl,\Om_{h+L})\big)\cap H^1\left(0,T;L^2(\Om_{h+L})^3\right),\\
    &\wit{\bm H}(x,t)\in L^2\big(0,T;H_0(\curl,\Om_{h+L})\big)\cap H^1\left(0,T;L^2(\Om_{h+L})^3\right),\\
    &\wit{\bm u}(x,t)\in L^2\big(0,T;H^1(\Om)^3\big)\cap H^1\left(0,T;L^2(\Om)^3\right),
\enn
with the stability estimate
\ben
    &\max\limits_{t\in[0,T]}\left(\|\pa_t \wit{\bm E}\|_{L^2(\Om_{h+L})^3}+\|\na\ti \wit{\bm E}\|_{L^2(\Om_{h+L})^3}\right.\\
    &\left.+\|\pa_t \wit{\bm H}\|_{L^2(\Om_{h+L})^3}+\|\na\ti \wit{\bm H}\|_{L^2(\Om_{h+L})^3}\right)\\
    &\lesssim{}(1+\sig_0T)^2\|\bm J\|_{H^1(0,T;L^2(\Om_h)^3)},
\enn
and
\ben
    \max\limits_{t\in[0,T]}(\|\pa_t \wit{\bm u}\|_{L^2(\Om)^3}+\|\na\cdot \wit{\bm u}\|_{L^2(\Om)}
    +\|\na\wit{\bm u}\|_{F(\Om)})
    \lesssim{}\sqrt{1+\sig_0T}\|\bm J\|_{L^1(0,T;L^2(\Om_h)^3)}.
\enn
\end{theorem}

\subsection{EtM operators for the PML problem}
Recalling the truncated PML problem (\ref{4.6}) in $s$-domain, let $\ch{\wit{\bm
 E}}=(\ch{\wit{E}}_1,\ch{\wit{E}}_2,\ch{\wit{E}}_3)^{\top}$ and $\ch{\wit{\bm
 H}}=(\ch{\wit{H}}_1,\ch{\wit{H}}_2,\ch{\wit{H}}_3)^{\top}$, denote by $\ch{\wit{\bm E}}_{\Ga_{h_j}}
 =-\n_j\ti(\n_j\ti\ch{\wit{\bm E}}|_{\Ga_{h_j}})=(\ch{\wit{E}}_1(\wit{x},h_j),\ch{\wit{E}}_2(\wit{x},h_j),0)^{\top}$
 and $\ch{\wit{\bm H}}\ti\n_j=(\ch{\wit{H}}_2(\wit{x},h_j),-\ch{\wit{H}}_1(\wit{x},h_j),0)^{\top}$
 the tangential component of the electric field and the tangential trace of the magnetic field on
 $\Ga_{h_j}$, respectively. We start by introducing the  EtM operators for the PML problem (\ref{4.6})
\ben
	\wit{\mathscr{B}}_{j}:H^{-1/2}(\curl,\Ga_{h_j})&\ra H^{-1/2}(\dive,\Ga_{h_j})\\
	\ch{\wit{\bm E}}_{\Ga_{h_j}}&\ra\ch{\wit{\bm H}}\ti\n_j
\enn
where $\ch{\wit{\bm E}}$ and $\ch{\wit{\bm H}}$ satisfy the following equations in the PML layer $\Om_{h_j}^{L_j}$
\begin{equation}\label{4.11}
    \begin{cases}
    \na\ti\ch{\wit{\bm E}}+\wit{\mu}_js\ch{\wit{\bm H}}=\0,&\gin\;\;\Om_{h_j}^{L_j}\\
    \na\ti\ch{\wit{\bm H}}-\wit{\vep}_js\ch{\wit{\bm E}}=\0,&\gin\;\;\Om_{h_j}^{L_j}\\
    \ch{\wit{\bm E}}(\wit{x},x_3)=\ch{\wit{\bm E}}(\wit{x},h_j),&\on\;\;\Ga_{h_j}\\
    \ch{\wit{\bm E}}\ti\n_j=\0,&\on\;\;\Ga_{h_j\pm L_j}.
    % \ch{\wit{\bm E}}\ti\n_2=\0.&\on\;\;\Ga_{h_2-L_2}
    \end{cases}
\end{equation}
Using the Maxwell's equations in (\ref{4.11}), we easily have
\begin{align}\label{4.11'}
  \ch{\wit{H}}_2(\wit{x},h_j)&=(\mu_js)^{-1}(\pa_{x_1}\ch{\wit{E}}_3-\pa_{x_3}\ch{\wit{E}}_1),\\ \label{4.11''}
  -\ch{\wit{H}}_1(\wit{x},h_j)&=(\mu_js)^{-1}(\pa_{x_2}\ch{\wit{E}}_3-\pa_{x_3}\ch{\wit{E}}_2).
\end{align}
Eliminating magnetic field $\wit{\bm H}$ from (\ref{4.11}) and writing it into component form, we obtain
\begin{subequations}\label{4.12}
\begin{alignat}{2}
        &\sig^{-1}\pa_{x_3}(\sig^{-1}\pa_{x_3}\ch{\wit{E}}_1)+\pa_{x_2}^2\ch{\wit{E}}_1 \no\\
        &-\pa_{x_1}[\pa_{x_2}\ch{\wit{E}}_2+\sig^{-1}\pa_{x_3}(\sig^{-1}\ch{\wit{E}}_3)]
        -s^2\mu_j\vep_j\ch{\wit{E}}_1=0,\\
        &\sig^{-1}\pa_{x_3}(\sig^{-1}\pa_{x_3}\ch{\wit{E}}_2)+\pa_{x_1}^2\ch{\wit{E}}_2 \no\\
        &-\pa_{x_2}[\pa_{x_1}\ch{\wit{E}}_1+\sig^{-1}\pa_{x_3}(\sig^{-1}\ch{\wit{E}}_3)]
        -s^2\mu_j\vep_j\ch{\wit{E}}_2=0,\\
        &\pa_{x_3}(\pa_{x_1}\ch{\wit{E}}_1+\pa_{x_2}\ch{\wit{E}}_2)-\pa_{x_1}^2\ch{\wit{E}}_3
        -\pa_{x_2}^2\ch{\wit{E}}_3+s^2\mu_j\vep_j\ch{\wit{E}}_3=0.
\end{alignat}
\end{subequations}
Noting that
\be\label{4.13}
  \na\cdot(\wit{\vep}_j\ch{\wit{\bm E}})=\vep_j(\sig\pa_{x_1}\ch{\wit{E}}_1+\sig\pa_{x_2}\ch{\wit{E}}_2
  +\pa_{x_3}(\sig^{-1}\ch{\wit{E}}_3))=0,
\en
then inserting (\ref{4.13}) into (\ref{4.12}) yields
\begin{subequations}\label{4.14}
\begin{alignat}{2}\label{4.14a}
        \sig^{-1}\pa_{x_3}(\sig^{-1}\pa_{x_3}\ch{\wit{E}}_1)+\pa_{x_1}^2\ch{\wit{E}}_1
        +\pa_{x_2}^2\ch{\wit{E}}_1-s^2\mu_j\vep_j\ch{\wit{E}}_1&=0,\\ \label{4.14b}
    \sig^{-1}\pa_{x_3}(\sig^{-1}\pa_{x_3}\ch{\wit{E}}_2)+\pa_{x_1}^2\ch{\wit{E}}_2
    +\pa_{x_2}^2\ch{\wit{E}}_2-s^2\mu_j\vep_j\ch{\wit{E}}_2&=0,\\ \label{4.14c}
    \pa_{x_3}[\sig^{-1}\pa_{x_3}(\sig^{-1}\ch{\wit{E}}_3)]+\pa_{x_1}^2\ch{\wit{E}}_3
    +\pa_{x_2}^2\ch{\wit{E}}_3-s^2\mu_j\vep_j\ch{\wit{E}}_3&=0.
\end{alignat}
\end{subequations}

For convenience, we only consider the derivation of EtM operator $\wit{\mathscr{B}}_{1}$ on $\Ga_{h_1}$. 
To do this, taking the Fourier transform of (\ref{4.14a}) and (\ref{4.14b}) with respect to $\wit{x}$ 
 leads to the ODEs
\be\label{4.15}
    \begin{cases}
        \pa_{\hat{x}_3}^2\wih{\ch{\wit{E}}}_j(\xi,x_3)-(\mu_1\vep_1s^2+|\xi|^2)\wih{\ch{\wit{E}}}_j(\xi,x_3)=0,
         &\gin\;\;\Om_{h_1}^{L_1}.\\
        \wih{\ch{\wit{E}}}_j(\xi,x_3)=\wih{\ch{\wit{E}}}_j(\xi,h_1),&\on \;\;\Ga_{h_1}\\
        \wih{\ch{\wit{E}}}_j(\xi,x_3)=0,&\on\;\;\Ga_{h_1+L_1}.
    \end{cases}
\en
The general solutions of ODEs (\ref{4.15}) can be easily represented as
\be\label{4.16}
    \wih{\ch{\wit{E}}}_j(\xi,x_3)=A_je^{\bt_1(\xi)(\hat{x}_3-h_1)}+B_je^{-\bt_1(\xi)(\hat{x}_3-h_1)},
    \;\;h_1<x_3<h_1+L_1.
\en
Letting $x_3=h_1$ and $x_3=h_1+L_1$ and applying the boundary conditions in (\ref{4.16}), respectively yields
\ben
	A_j=-\frac{e^{-\bt_1(\xi)\wit{L}_1\wih{\ch{\wit{E}}}_j(\xi,h_1)}}
	{e^{\bt_1(\xi)\wit{L}_1}-e^{-\bt_1(\xi)\wit{L}_1}},\;\;\;\; 
	B_j=\frac{e^{\bt_1(\xi)\wit{L}_1\wih{\ch{\wit{E}}}_j(\xi,h_1)}}
	{e^{\bt_1(\xi)\wit{L}_1}-e^{-\bt_1(\xi)\wit{L}_1}},
\enn
where
\be\label{4.16+}
    \wit{L}_1:=\int_{h_1}^{h_1+L_1}\sig(\tau)d\tau=L_1+\frac{s_1^{-1}}{m+1}L_1\sig_1.
\en
Hence, the solution of (\ref{4.15}) is described as
\be\label{4.16++}
    \wih{\ch{\wit{E}}}_j(\xi,x_3)=\frac{e^{-\bt_1(\xi)(\hat{x}_3-h_1-\wit{L}_1)}-e^{\bt_1(\xi)(\hat{x}_3-h_1-\wit{L}_1)}
    }{e^{\bt_1(\xi)\wit{L}_1}-e^{-\bt_1(\xi)\wit{L}_1}}\wih{\ch{\wit{E}}}_j(\xi,h_1),\;\;h_1<x_3<h_1+L_1.
\en
Taking the normal derivative of (\ref{4.16++}) and evaluate the value on $\Ga_{h_1}$, we obtain
\be\label{normal}
    \frac{\pa\wih{\ch{\wit{E}}}_j(\xi,h_1)}{\pa x_3}=-\bt_1(\xi)\coth[\bt_1(\xi)\wit{L}_1]\wih{\ch{\wit{E}}}_j(\xi,h_1),
\en
where $\coth(t):=\frac{e^t+e^{-t}}{e^t-e^{-t}}$ denotes the hyperbolic cotangent function and the fact that
 $\sig=1$ on $\Ga_{h_1}$ has been used.

Next, we consider the equation (\ref{4.14c}). Let $P=\sig^{-1}\ch{\wit{E}}_3$,
by divergence free condition (\ref{4.13}) and PEC
boundary condition on $\Ga_{h_1+L_1}$, we have
\ben
  \pa_{x_3}P(\wit{x},x_3)=0,\;\;\on\;\;\Ga_{h_1+L_1}.
\enn
Taking the Fourier transform of (\ref{4.14c}) with respect to $\wit{x}$, we obtain
\be\label{4.19}
    \begin{cases}
        \pa_{\hat{x}_3}^2\wih{P}(\xi,x_3)-(\mu_1\vep_1s^2+|\xi|^2)\wih{P}(\xi,x_3)=0,
         &\gin\;\;\Om_{h_1}^{L_1}.\\
        \wih{P}(\xi,x_3)=\wih{\ch{\wit{E}}}_3(\xi,h_1),&\on \;\;\Ga_{h_1}\\
        \pa_{\hat{x}_3}\wih{P}(\xi,x_3)=0.&\on\;\;\Ga_{h_1+L_1}
    \end{cases}
\en
Similarly, we get the general solution of (\ref{4.19}) that
\be\label{4.20}
    \wih{P}(\xi,x_3)=\frac{e^{-\bt_1(\xi)(\hat{x}_3-h_1-\wit{L}_1)}+e^{\bt_1(\xi)(\hat{x}_3-h_1-\wit{L}_1)}
    }{e^{\bt_1(\xi)\wit{L}_1}+e^{-\bt_1(\xi)\wit{L}_1}}\wih{\ch{\wit{E}}}_3(\xi,h_1),\;\;h_1<x_3<h_1+L_1.
\en
Taking the normal derivative of (\ref{4.20}) and evaluate the value on $\Ga_{h_1}$, we obtain
\ben
    \frac{\pa\wih{P}(\xi,h_1)}{\pa x_3}=-\frac{\bt_1(\xi)}{\coth[\bt_1(\xi)\wit{L}_1]}\wih{\ch{\wit{E}}}_3(\xi,h_1),
\enn
It follows from (\ref{4.13}) and $\sig=1$ on $\Ga_{h_1}$ that
\ben
    \wih{\ch{\wit{E}}}_3(\xi,h_1)&=-\frac{\coth[\bt_1(\xi)\wit{L}_1]\pa_{x_3}\wih{P}(\xi,h_1)}{\bt_1(\xi)}\\
    &=\frac{\coth[\bt_1(\xi)\wit{L}_1]i}{\bt_1(\xi)}\left(\xi_1\wih{\ch{\wit{E}}}_1(\xi,h_1)
    +\xi_2\wih{\ch{\wit{E}}}_2(\xi,h_1)\right).
\enn
This, combining (\ref{4.11'})-(\ref{4.11''}) and (\ref{normal}) leads to
\ben
  \wih{\ch{\wit{H}}}_2(\xi,h_1)&=\frac{1}{\mu_1s}\left[i\xi_1\wih{\ch{\wit{E}}}_3(\xi,h_1)
  -\pa_{x_3}\wih{\ch{\wit{E}}}_1(\xi,h_1)\right]\\
  &=\frac{\coth[\bt_1(\xi)\wit{L}_1]}{\mu_1s}\left[-\frac{\xi_1}{\bt_1(\xi)}
  \left(\xi_1\wih{\ch{\wit{E}}}_1(\xi,h_1)+\xi_2\wih{\ch{\wit{E}}}_2(\xi,h_1)\right)
  +\bt_1(\xi)\wih{\ch{\wit{E}}}_1(\xi,h_1)\right]\\
  &=\frac{\coth[\bt_1(\xi)\wit{L}_1]}{\mu_1s\bt_1(\xi)}\left[\vep_1\mu_1s^2\wih{\ch{\wit{E}}}_1(\xi,h_1)
  +\xi_2\left(\xi_1\wih{\ch{\wit{E}}}_1(\xi,h_1)-\xi_2\wih{\ch{\wit{E}}}_2(\xi,h_1)\right)\right],
\enn
and
\ben
  -\wih{\ch{\wit{H}}}_1(\xi,h_1)&=\frac{1}{\mu_1s}\left[i\xi_2\wih{\ch{\wit{E}}}_3(\xi,h_1)
  -\pa_{x_3}\wih{\ch{\wit{E}}}_2(\xi,h_1)\right]\\
  &=\frac{\coth[\bt_1(\xi)\wit{L}_1]}{\mu_1s}\left[-\frac{\xi_2}{\bt_1(\xi)}
  \left(\xi_1\wih{\ch{\wit{E}}}_1(\xi,h_1)+\xi_2\wih{\ch{\wit{E}}}_2(\xi,h_1)\right)
  +\bt_1(\xi)\wih{\ch{\wit{E}}}_2(\xi,h_1)\right]\\
  &=\frac{\coth[\bt_1(\xi)\wit{L}_1]}{\mu_1s\bt_1(\xi)}\left[\vep_1\mu_1s^2\wih{\ch{\wit{E}}}_2(\xi,h_1)
  +\xi_1\left(\xi_1\wih{\ch{\wit{E}}}_2(\xi,h_1)-\xi_2\wih{\ch{\wit{E}}}_1(\xi,h_1)\right)\right].
\enn

Now, for any tangential vector $\bm\om=(\om_1,\om_2,0)^{\top}$ defined on $\Ga_{h_1}$,
we obtain the explicit representation of the EtM operator $\wit{\mathscr{B}}_{1}$
\be\label{dtn1_pml}
    \wit{\mathscr{B}}_{1}\bm\om=(v_1,v_2,0)^{\top},
\en
where
\ben
	\wih{v}_1
	&=\frac{\coth[\bt_1(\xi)\wit{L}_1]}{\mu_1s\bt_1(\xi)}\left[\vep_1\mu_1s^2\wih{\om}_1
	+\xi_2(\xi_2\wih{\om}_1-\xi_1\wih{\om}_2)\right],\\
	\wih{v}_2
	&=\frac{\coth[\bt_1(\xi)\wit{L}_1]}{\mu_1s\bt_1(\xi)}\left[\vep_1\mu_1s^2\wih{\om}_2
	+\xi_1(\xi_1\wih{\om}_2-\xi_2\wih{\om}_1)\right],
\enn
with
\be\label{thick1}
	\wit{L}_1:=\int_{h_1}^{h_1+L_1}\sig(\tau)d\tau=L_1+\frac{s_1^{-1}}{m+1}L_1\sig_1.
\en
Similarly, for any tangential vector $\bm\om=(\om_1,\om_2,0)^{\top}$ defined on $\Ga_{h_2}$, 
the EtM operator $\wit{\mathscr{B}}_{2}$ has the following form
\be\label{dtn2_pml}
	\wit{\mathscr{B}}_{2}\bm\om=(v_1,v_2,0)^{\top},
\en
where
\ben
	\wih{v}_1
	&=\frac{\coth[\bt_2(\xi)\wit{L}_2]}{\mu_2s\bt_2(\xi)}\left[\vep_2\mu_2s^2\wih{\om}_1
	+\xi_2(\xi_2\wih{\om}_1-\xi_1\wih{\om}_2)\right],\\
	\wih{v}_2
	&=\frac{\coth[\bt_2(\xi)\wit{L}_2]}{\mu_2s\bt_2(\xi)}\left[\vep_2\mu_2s^2\wih{\om}_2
	+\xi_1(\xi_1\wih{\om}_2-\xi_2\wih{\om}_1)\right],
\enn
with
\be\label{thick2}
	\wit{L}_2:=\int_{h_2-L_2}^{h_2}\sig(\tau)d\tau=L_2+\frac{s_1^{-1}}{m+1}L_2\sig_2.
\en

We now find that the truncated PML problem (\ref{4.7}) is equivalently reduced to the
following boundary value problem
\be\label{4.17}
    \begin{cases}
     \na\ti((s\wit{\mu})^{-1}\na\ti\ch{\wit{\bm E}})+s\wit{\vep}\ch{\wit{\bm E}}=-\ch{\bm J},
     &\gin\;\;\Om_{h+L}\\
     \De^*\ch{\wit{\bm u}}-\rho_is^2\ch{\wit{\bm u}}=\0,& \gin\;\; \Om\\
      \n\ti[\ch{\wit{\bm E}}]=\n\ti[\mu^{-1}\na\ti\ch{\wit{\bm E}}]=\0,&\on\;\;\Ga_f\\
    -\n\ti(s\wit{\mu})^{-1}\na\ti\ch{\wit{\bm E}}=\bm{T\ch{\wit{\bm u}}},&\on\;\; \Ga\\
    \n\ti\ch{\wit{\bm E}}=\n\ti s\ch{\wit{\bm u}},&\on\;\; \Ga\\
   (s\mu_j)^{-1}\na\ti\ch{\wit{\bm E}}\ti\n_j+\wit{\mathscr{B}}_j[\ch{\wit{\bm E}}_{\Ga_{h_j}}]=\0,
   &\on\;\;\Ga_{h_j},\quad j=1,2.
    \end{cases}
\en
The variational formulation of (\ref{4.17}) can be obtained: 
to find $(\ch{\wit{\bm E}},\ch{\wit{\bm u}})\in\mathscr{X}_s$ such that
\be\label{4.18}
    a_{p}\big((\ch{\wit{\bm E}},\ch{\wit{\bm u}}),(\bm V,\bm v)\big)=-\int_{\Om_{h}}\ch{\bm J}\cdot\ov{\bm V}dx
    \;\;\;\;\text{for all}\;\;(\bm V,\bm v)\in \mathscr{X}_s,
\en
where the sesquilinear form $a_p(\cdot,\cdot)$ is defined as
\begin{align}\label{4.18a}
    a_p\big((\ch{\wit{\bm{E}}},\ch{\wit{\bm{u}}}),(\bm{V},\bm{v})\big)&=\int_{\Om_h}\big((s\wit{\mu})^{-1}
    (\na\ti\ch{\wit{\bm{E}}})\cdot(\na\ti\ov{\bm V})dx+s\wit{\vep}\ch{\wit{\bm{E}}}\cdot\ov{\bm V}\big)dx  \\\no
    &+\int_\Om\Big[\ov{s}\mathcal{E}(\ch{\wit{\bm{u}}},\ov{\bm v})+\rho_is|s|^2\ch{\wit{\bm{u}}}\cdot\ov{\bm v}\Big]dx
    +\sum_{j=1}^2\langle\wit{\mathscr{B}}_j[\ch{\wit{\bm{E}}}_{\Ga_{h_j}}],\bm{V}_{\Ga_{h_j}}\rangle_{\Ga_{h_j}}.
\end{align}

\subsection{Exponential convergence of the time domain PML solution}
In this section, we shall give an error estimate between the solution $(\bm E,\bm u)$ of
the original equations (\ref{2.7}) and the solution $(\wit{\bm E},\wit{\bm u})$ of
the truncated PML problem (\ref{4.10}). The following fundamental Lemma on the error estimate 
between the EtM operators $\mathscr{B}_j$ and the EtM operators $\wit{\mathscr{B}}_j$ 
is essential to the exponential convergence of the PML method.

\begin{lem}\label{thm4.3}
For $j=1,2$, denote $\ov{L}_j=\frac{L_j\sig_j}{m+1}$. Then for $s=s_1+is_2$ with $s_1>0$, 
we have the following estimate
\ben
	\|\mathscr{B}_j-\wit{\mathscr{B}}_{j} \|_{L(H^{-1/2}(\curl,\Ga_{h_j}),H^{-1/2}(\dive,\Ga_{h_j}))}
	\leq\Ga_j\frac{2e^{-2\sqrt{\vep_j\mu_j}\ov{L}_j}}{1-e^{-2\sqrt{\vep_j\mu_j}\ov{L}_j}}:=M_j,
\enn
where $\Ga_j$ is defined in \eqref{4.24''}, and $L(X,Y)$ denotes the standard space of the 
bounded linear operators from the Hilbert space $X$ to the Hilbert space $Y$.
\end{lem}

\begin{proof}
Given $\bm u=(u_1,u_2,0)^{\top},\;\bm v=(v_1,v_2,0)^{\top}\in H^{-1/2}(\curl,\Ga_{h_j})$,
we have from the definitions of $\mathscr{B}_j$ (see (\ref{dtn})) and $\wit{\mathscr{B}}_j$
 (see (\ref{dtn1_pml}) and (\ref{dtn2_pml})) that
\begin{align}\label{4.23''}
	&\;\;\;\;\langle(\mathscr{B}_j-\wit{\mathscr{B}}_{j})\bm u,\bm v\rangle_{\Ga_{h_j}}\no\\
	&=\int_{\R^2}\frac{(1+|\xi|^2)^{1/2}}{\mu_js\bt_j(\xi)}(1-\coth[\bt_j(\xi)\wit{L}_j])(1+|\xi|^2)^{-1/2}\no\\
	&\quad\left[\vep_j\mu_js^2(\wih{u}_1\ov{\wih{v}}_1+\wih{u}_2\ov{\wih{v}}_2)
	+(\xi_1\wih{u}_2-\xi_2\wih{u}_1)\cdot(\xi_1\ov{\wih{v}}_2-\xi_2\ov{\wih{v}}_1)\right]d\xi.
\end{align}
Hence we need to estimate the term
\ben
  \frac{(1+|\xi|^2)^{1/2}}{|\bt_j(\xi)|}\left|1-\coth[\bt_j(\xi)\wit{L}_j]\right|.
\enn
Firstly, we denote
\ben
	\vep_j\mu_js^2 = a_j + i b_j,\;{\rm with}\;a_j = \vep_j\mu_j(s_1^2-s_2^2),
	\;b_j = 2\vep_j\mu_js_1s_2,
\enn
and
\ben
	\bt_j^2 = \vep_j\mu_js^2
	+ |\xi|^2 = \phi_j + i b_j,\;{\rm with}\;\phi_j = \Rt(\vep_j\mu_js^2) + |\xi|^2 = a_j + |\xi|^2.
\enn
Noting that
\ben
	\frac{(1+|\xi|^2)^{1/2}}{|\bt_j(\xi)|} = \Big[\frac{(1+\phi_j-a_j)^2}{\phi_j^2+b_j^2}\Big]^{1/4},
\enn
we define an auxiliary function
\ben
	F_j(t) = \frac{(1+t-a_j)^2}{t^2+b_j^2},\;t\geq a_j.
\enn
Simple calculations gives the derivative
\ben
	F_j^{\prime}(t)=\frac{2(t-a_j+1)[(a_j-1)t+b_j^2]}{(t^2+b_j^2)^2}.
\enn
We consider the following two cases:

\begin{enumerate}[(I)]

\item If $s_2^2\geq s_1^2$, then $a_j\leq 0$. Setting $K_j:=\frac{b_j^2}{1-a_j}$, 
it can be verified that $F_j(t)$ increases in $[a_j,K_j]$, and decreases in $[K_j,+\infty)$. 
Hence $F_j(t)$ reaches its maximum $\frac{(1-a_j)^2+b_j^2}{b_j^2}$ at $K_j$.
 
\item If $s_2^2<s_1^2$, then $a_j>0$. We have another three possibilities.

\item[(II.a)] $1-a_j<0$, then $F_j(t)$ increases in $[a_j,+\infty)$, hence
\ben
	F_j(t)\leq{\lim_{t \to +\infty}}F_j(t)=1.
\enn

\item[(II.b)] $1-a_j=0$, it can be easily verified that
\ben
	F_j(t)=\frac{t^2}{t^2+b_j^2}\leq 1.
\enn

\item[(II.c)] $1-a_j>0$, that is $1-\vep_j\mu_js_1^2+\vep_j\mu_js_2^2>0$. 
In this case, we need to compare the size of $a_j$ and $K_j$. Note that
 $K_j\leq a_j$ is equivalent to
\ben
	s_2^4 + \left(2s_1^2 + \frac{1}{\vep_j\mu_j}\right)s_2^2 + 
	s_1^2\left(s_1^2-\frac{1}{\vep_j\mu_j}\right) \leq 0.
\enn
Thus define
\be\label{epsion}
	\vep_0(s_1):=-(s_1^2+\frac{1}{2\vep_j\mu_j})+\sqrt{\frac{2s_1^2}
	{\vep_j\mu_j}+\frac{1}{4\vep_j^2\mu_j^2}}.
\en
We further have three cases:
\item[(II.c.i)] $1-\vep_j\mu_js_1^2<0$, then $s_2^2>\frac{\vep_j\mu_js_1^2-1}
{\vep_j\mu_j}>0$ and $\vep_0(s_1)<0$, then $a_j<K_j$.
Hence
\ben
	F_j(t)\leq F_j(K_j)=\frac{(1-a_j)^2+b_j^2}{b_j^2}.
\enn

\item[(II.c.ii)] $1-\vep_j\mu_js_1^2=0$, then $s_2^2>0$ and $\vep_0(s_1)=0$, 
it holds that $F_j(t)\leq F_j(K_j)=1+\frac{s_2^2}{4s_1^2}$.

\item[(II.c.iii)] $1-\vep_j\mu_js_1^2>0$, then we have the following two cases:

\item[(II.c.iii.1)] If $s_2^2\leq\vep_0(s_1)$, then $K_j\leq a_j$, therefore $F_j(t)$ 
decreases in $[a_j,+\infty)$, then
\ben
	F_j(t)\leq F_j(a_j)=\frac{1}{a_j^2+b_j^2}.
\enn

\item[(II.c.iii.2)] If $s_2^2>\vep_0(s_1)$, then $K_j>a_j.$ Hence
\ben
	F_j(t)\leq F_j(K_j)=\frac{(1-a_j)^2+b_j^2}{b_j^2}.
\enn

\end{enumerate}
Recalling the definitions of $a_j$ and $b_j$, by the above discussions, we arrive at
\be\label{4.23'}
	\frac{(1+|\xi|^2)^{1/2}}{|\bt_j(\xi)|}\leq\Lambda_j(s_1,s_2),
\en
where $\Lambda_j(s_1,s_2)$ is defined as:\\
(1) when $1-\vep_j\mu_js_1^2< 0$,
\begin{equation*}
\Lambda_j(s_1,s_2)=
\left\{
\begin{aligned}
&1, &0\leq s_2^2\leq s_1^2-\frac{1}{\vep_j\mu_j}, \\
 &\Big[1+\frac{\big(1-\vep_j\mu_j(s_1^2-s_2^2)\big)^2}{4\vep_j^2\mu_j^2s_1^2s_2^2}\Big]^{1/4},
 & s_2^2>s_1^2-\frac{1}{\vep_j\mu_j}.
\end{aligned}
\right.
\end{equation*}
(2) when $1-\vep_j\mu_js_1^2 = 0$,
\[\Lambda_j(s_1,s_2)=\big(1+\frac{s_2^2}{4s_1^2}\big)^{1/4}.
\]
(3) when $1-\vep_j\mu_js_1^2> 0$,
\begin{equation*}
\Lambda_j(s_1,s_2)=
\left\{
\begin{aligned}
&\frac{1}{\sqrt{\vep_j\mu_j}|s|},   &0\leq s_2^2\leq \vep_0(s_1), \\
 &\Big[1+\frac{\big(1-\vep_j\mu_j(s_1^2-s_2^2)\big)^2}{4\vep_j^2\mu_j^2s_1^2s_2^2}\Big]^{1/4},
 & s_2^2>\vep_0(s_1).
\end{aligned}
\right.
\end{equation*}
In the following, we further estimate
\begin{align}\label{4.24}
    \sup\limits_{\xi\in\R^2}\left|1-\coth[\bt_j(\xi)\wit{L}_j]\right|
    ={}&\sup\limits_{\xi\in\R^2}\frac{\left|2e^{-2\bt_{j_r}(\xi)\wit{L}_j}\right|}
    {\left|1-e^{-2(\bt_{j_r}(\xi)+i\bt_{j_i}(\xi))\wit{L}_j}\right|}\no\\
    \leq&\sup\limits_{\xi\in\R^2}\frac{2e^{-2\bt_{j_r}(\xi)\wit{L}_j}}{1-e^{-2\bt_{j_r}(\xi)\wit{L}_j}},
\end{align}
where $\bt_{j_r}(\xi)=\Rt[\bt_j(\xi)]$, and $\bt_{j_i}(\xi)=\I[\bt_j(\xi)]$.
By the formulas
\begin{align*}
    z^{1/2}=\sqrt{\frac{|z|+z_1}{2}}+i{\rm sgn}(z_2)\sqrt{\frac{|z|-z_1}{2}}
    ,\;\;\text{for}\;\; z=z_1+iz_2,\; \Rt[z^{1/2}]>0,
\end{align*}
we have
\begin{align*}
    \bt_{j_r}(\xi)&=\sqrt{\frac{|\bt_j^2(\xi)|+\Rt[\bt_j^2(\xi)]}{2}}\\
    &=\left[\frac{[(\vep_j\mu_j(s_1^2-s_2^2)+|\xi|^2)^2+4\vep_j^2\mu_j^2s_1^2s_2^2]^{1/2}
    +\vep_j\mu_j(s_1^2-s_2^2)+|\xi|^2}{2}\right]^{1/2}.
\end{align*}

Note that $\frac{2e^{-2\bt_{j_r}(\xi)\wit{L}_j}}{1-e^{-2\bt_{j_r}(\xi)\wit{L}_j}}$ is 
monotonically decreasing with respect to $\bt_{j_r}(\xi)$. Hence, we need to seek the maximum of $\bt_{j_r}(\xi)$ in $\R^2$.
Simple calculations yields that $\xi=0$ is the unique extreme point of the function $\bt_{j_r}(\xi)$, and
\begin{align*}
  \bt_{j_r}(0)=\sqrt{\vep_j\mu_j}s_1,\frac{2e^{-2\bt_{j_r}(\xi)\wit{L}_j}}
  {1-e^{-2\bt_{j_r}(\xi)\wit{L}_j}}\Big|_{\xi=0}
  =\frac{2e^{-2\sqrt{\vep_j\mu_j}s_1\wit{L}_j}}{1-e^{-2\sqrt{\vep_j\mu_j}s_1\wit{L}_j}}.
\end{align*}
Besides, $\bt_{j_r}(\xi)\ra +\infty$, thereby, $\frac{2e^{-2\bt_{j_r}(\xi)\wit{L}_j}}
{1-e^{-2\bt_{j_r}(\xi)\wit{L}_j}}\ra 0$, as $\xi\ra\infty$.

By the definitions of $\wit{L}_1$ and $\wit{L}_2$ (see (\ref{thick1}) and (\ref{thick2})), 
we therefore conclude that
\begin{align}\label{4.24'}
    \sup\limits_{\xi\in\R^2}\frac{2e^{-2\bt_{j_r}(\xi)\wit{L}_j}}{1-e^{-2\bt_{j_r}(\xi)\wit{L}_j}}
    ={}&\frac{2e^{-2\sqrt{\vep_j\mu_j}s_1\wit{L}_j}}{1-e^{-2\sqrt{\vep_j\mu_j}s_1\wit{L}_j}}
    \leq\frac{2e^{-2\sqrt{\vep_j\mu_j}\ov{L}_j}}{1-e^{-2\sqrt{\vep_j\mu_j}\ov{L}_j}}.
\end{align}
Combining (\ref{4.23'}) and(\ref{4.24'}) as well as Cauchy-Schwartz inequality for (\ref{4.23''}) yields
\begin{align*}
  |\langle(\mathscr{B}_j-\wit{\mathscr{B}}_{j})\bm u,\bm v\rangle_{\Ga_{h_j}}|
  \leq \Ga_j\frac{2e^{-2\sqrt{\vep_j\mu_j}\ov{L}_j}}{1-e^{-2\sqrt{\vep_j\mu_j}\ov{L}_j}}
  \Vert\bm u\Vert_{H^{-1/2}(\curl,\Ga_{h_j})}\Vert\bm v\Vert_{H^{-1/2}(\curl,\Ga_{h_j})},
\end{align*}
where
\begin{align}\label{4.24''}
  \Ga_j=\frac{1}{\mu_j|s|}\Lambda_j(s_1,s_2)
  \max\{\vep_j\mu_j|s|^2,1\}.
\end{align}
This completes the proof.
\end{proof}

Let $\bm\om=(\ch{\bm E},\ch{\bm u})$ and $\bm\om_p=(\ch{\wit{\bm E}},\ch{\wit{\bm u}})$ 
be the solutions of the variational problems (\ref{3.4}) and (\ref{4.18}), respectively. 
By the definitions of  variational formulations of $a(\cdot,\cdot)$ and $a_p(\cdot,\cdot)$, 
we obtain
\begin{align}\label{4.25}
   &\quad\;|a(\bm\om-\bm\om_p,\bm\om-\bm\om_p)| \no\\
&=|a(\bm\om,\bm\om-\bm\om_p)-a(\bm\om_p,\bm\om-\bm\om_p)|\no\\
&=|a_p(\bm\om_p,\bm\om-\bm\om_p)-a(\bm\om_p,\bm\om-\bm\om_p)|\no\\
&=\Big|\sum_{j=1}^{2}\langle(\mathscr{B}_j-\wit{\mathscr{B}}_{j})
[\check{\wit{\bm E}}_{\Ga_{h_j}}],(\check{\bm E}-
\check{\wit{\bm E}})_{\Ga_{h_j}}\rangle_{\Ga_{h_j}}\Big|\no\\
&\leq\eta^2\sum_{j=1}^{2}\|\mathscr{B}_j-\wit{\mathscr{B}}_{j}
        \|_{L(H^{-1/2}(\curl,\Ga_{h_j}),H^{-1/2}(\dive,\Ga_{h_j}))}\|\bm\om_p\|_{\mathscr{X}_s}
        \|\bm\om-\bm\om_p\|_{\mathscr{X}_s},
\end{align}
where the constant $\eta=\max\{\sqrt{1+(h_1-h_2)^{-1}},\sqrt{2}\}$ is defined in Lemma \ref{lemB.2}.
Now we arrive at our main theorem by concluding the above argument.
\begin{theorem}\label{convergence}
Let $(\bm E,\bm u)$ be the solution of problem \eqref{2.7}, and $(\wit{\bm E},\wit{\bm u})$
 be the solution of problem \eqref{4.10} with $s_1=1/T$, $\sig_0=\max\{\sig_1,\sig_2\}$, 
 then under the assumptions \eqref{assumption} and \eqref{assumption1}
we have the following error estimate
  \begin{align}\label{result}
    &\int_{0}^{T}(\|\bm E-\wit{\bm E}\|_{H(\curl,\Om_h)}^2+\|\bm u-\wit{\bm u}\|_{H^1(\Om)^3}^2)dt\\\no
    \lesssim{}&\max\{1,T^2\}(T^4+2T^2)(\g_1+\g_2)(1+\sig_0T)^2
    \Big(\sum_{j=1}^{2}\frac{2e^{-\sqrt{\vep_j\mu_j}\sigma_jL_j}}
    {1-e^{-\sqrt{\vep_j\mu_j}\sigma_jL_j}}\Big)^2\|\bm J\|^2_{H^5(0,T;L^2(\Om_h)^3)},
\end{align}
where $\g_1$ and $\g_2$ are positive constants independent of  $(\bm E,\bm u)$ 
and $(\wit{\bm E},\wit{\bm u})$, but that may depend on $T$.
\end{theorem}

\begin{proof}
Combining (\ref{4.25}) with Lemma \ref{thm4.3} and the uniform coercivity
 (\ref{coercivity}) of $a(\cdot,\cdot)$, we have
\ben
    \|\bm\om-\bm\om_p\|_{\mathscr{X}_s}
    \leq{}C^{-1}\eta^2(M_1+M_2)\|\bm\om_p\|_{\mathscr{X}_s}.
\enn
By the Parseval identity (\ref{A.5}) and the definitions of $M_1, M_2$ in Lemma \ref{thm4.3}, we get
\begin{align*}
    \int_{0}^{\ify}e^{-2s_1t}\|\mathscr{L}^{-1}(\bm\om-\bm\om_p)\|_{\mathscr{X}_s}^2dt
    ={}&\frac{1}{2\pi}\int_{-\ify}^{\ify}\|\bm\om-\bm\om_p\|_{\mathscr{X}_s}^2ds_2\no\\
    \leq{}&\frac{1}{2\pi}\int_{-\ify}^{\ify}
    C^{-2}\eta^4\Big(\sum_{j=1}^{2}\Ga_j\frac{2e^{-2\sqrt{\vep_j\mu_j}\ov{L}_j}}
    {1-e^{-2\sqrt{\vep_j\mu_j}\ov{L}_j}}\Big)^2\|\bm\om_p\|_{\mathscr{X}_s}^2ds_2.
\end{align*}
This implies that
\begin{equation}\label{4.26-}
\begin{aligned}
  &\int_0^T(\|\bm E-\wit{\bm E}\|_{H(\curl,\Om_h)}^2+\|\bm u-\wit{\bm u}\|_{H^1(\Om)^3}^2)dt\\
  \leq{}&e^{2s_1T}\int_{0}^{\ify}e^{-2s_1t}(\|\bm E-\wit{\bm E}\|_{H(\curl,\Om_h)}^2+\|\bm u-\wit{\bm u}\|_{H^1(\Om)^3}^2)dt\\
  ={}&e^{2s_1T} \int_{0}^{\ify}e^{-2s_1t}\|\mathscr{L}^{-1}(\bm\om-\bm\om_p)\|_{\mathscr{X}_s}^2dt\\
  \leq{}&\frac{\eta^4e^{2s_1T}}{\pi}\int_{0}^{\ify}C^{-2}\Big(\sum_{j=1}^{2}\Ga_j\frac{2e^{-2\sqrt{\vep_j\mu_j}\ov{L}_j}}
   {1-e^{-2\sqrt{\vep_j\mu_j}\ov{L}_j}}\Big)^2\|\bm\om_p\|_{\mathscr{X}_s}^2ds_2.
\end{aligned}
\end{equation}
Since $s_1>0$ is arbitrarily fixed, recalling the definitions of $C$ in (\ref{coercivity}) 
and $\Ga_j$ in (\ref{4.24''}), there exists a sufficiently large positive constant $M$, such that
\begin{align}\label{4.26}
C^{-2}\Ga_j^2, C^{-2}\Ga_1\Ga_2\leq\g_1|s|^{8},
\end{align}
when $s_2\geq M$, where $\g_1$ is a constant independence of $s_2$. 
On the other hand, it's clear that
\begin{align}\label{4.27}
 C^{-2}\Ga_j^2, C^{-2}\Ga_1\Ga_2\leq\g_2,
\end{align}
when $0\leq s_2\leq M$,  where $\g_2$ is a constant independence of $s_2$.
Thus the last inequality in (\ref{4.26-}) becomes
\be\label{4.26+}
&\quad\int_{0}^{\ify}C^{-2}\Big(\sum_{j=1}^{2}\Ga_j\frac{2e^{-2\sqrt{\vep_j\mu_j}\ov{L}_j}}
{1-e^{-2\sqrt{\vep_j\mu_j}\ov{L}_j}}\Big)^2\|\bm\om_p\|_{\mathscr{X}_s}^2ds_2\\
&\leq\Big(\sum_{j=1}^{2}\frac{2e^{-2\sqrt{\vep_j\mu_j}\ov{L}_j}}
{1-e^{-2\sqrt{\vep_j\mu_j}\ov{L}_j}}\Big)^2
\left(\int_{0}^{M}\g_2\|\bm\om_p\|_{\mathscr{X}_s}^2ds_2
+\int_{M}^{\ify}\g_1\|s^4\bm\om_p\|_{\mathscr{X}_s}^2ds_2\right).
\en
Now, only the right-hand integral in (\ref{4.26+}) remains to be estimated. 
Combining Lemma \ref{thm4.1} with Parseval identity (\ref{A.5}) and the assumptions 
(\ref{assumption})-(\ref{assumption1}) yields
\begin{equation*}
\begin{aligned}
  &\int_{0}^{M}\g_2\|\bm\om_p\|_{\mathscr{X}_s}^2ds_2+\int_{M}^{\ify}\g_1\|s^4\bm\om_p\|_{\mathscr{X}_s}^2ds_2\\
  \leq{}&(1+s_1^{-1}\sig_0)^2
    \Big(\int_{0}^{M}\g_2\big[\frac{1+2s_1^2}{s_1^4\min\{1,s_1^2\}}
    \|\ch{\bm J}\|_{L^2(\Om_h)^3}^2+s_1^{-2}\| s\ch{\bm J}\|_{L^2(\Om_h)^3}^2\big]ds_2\\
    &+\int_{M}^{\ify}\g_1\big[\frac{1+2s_1^2}{s_1^4\min\{1,s_1^2\}}
    \| s^4\ch{\bm J}\|_{L^2(\Om_h)^3}^2+s_1^{-2}\| s^5\ch{\bm J}\|_{L^2(\Om_h)^3}^2\big]ds_2\Big)\\
  \leq{}&\frac{1+2s_1^2}{s_1^4\min\{1,s_1^2\}}(\g_1+\g_2)(1+s_1^{-1}\sig_0)^2
  \int_{0}^{\ify}\sum_{l=0}^{5}\| s^l\ch{\bm J}\|_{L^2(\Om_h)^3}^2ds_2\\
  ={}&\pi\frac{1+2s_1^2}{s_1^4\min\{1,s_1^2\}}(\g_1+\g_2)(1+s_1^{-1}\sig_0)^2
  \int_{0}^{\ify}\sum_{l=0}^{5}\| \pa_t^l\bm J\|_{L^2(\Om_h)^3}^2dt.
\end{aligned}
\end{equation*}
By this inequality and (\ref{4.26-}), (\ref{4.26+}) the required estimate (\ref{result}) follows easily
on taking $s_1=T^{-1}$ and using the assumption (\ref{assumption1}) again, where integer $m\geq 1$
should be chosen small enough to ensure the rapid convergence (thus we need to take $m = 1$)
noting the definition of $\ov{L}_j=\sig_jL_j/(m+1)$. The proof is thus complete.

\end{proof}

\begin{remark}\label{re5} {\rm
Theorem \ref{convergence} implies that, for large $T$ the exponential convergence of the PML method
can be achieved by enlarging the thickness $L_j$ or the PML absorbing parameter $\sigma_j$
which increases as $\ln T$.
}
\end{remark}

\section{Conclusions}\label{sec6}

In this paper, the scattering of a time-dependent electromagnetic wave by an an elastic body
immersed in the lower half-space of a two-layered background medium is studied.
The well-posedness and stability estimate is verified by using the Laplace transform, 
the variational method and the energy method. In addition, we propose an effective PML method
 to solve this interaction problem, based on a real coordinate stretching technique associated 
 with $[\Rt(s)]^{-1}$ in the frequency domain, where $s$ is the Laplace transform variable.
The well-posedness and stability of the truncated PML problem are proved by using the Laplace transform
and energy method. At last, through the error estimate between the EtM operators of the original problem
and the  EtM operators for the PML problem, we establish the exponential convergence depending on the 
thickness and parameters of the PML layers.

In practical computation, the PML medium must be truncated along the lateral direction which
may be achieved by constructing the rectangular or cylindrical PML. Further, the idea of real 
coordinate stretching could be extended to other time-dependent scattering problems, 
such as diffraction gratings, elastic rough surface scattering problems. 
We hope to report such results in the future.

\begin{appendix}
\renewcommand{\theequation}{\thesection.\arabic{equation}}  %更改公式编号为A.1, A.2...

\section{Laplace transform}\label{ap1}
For each $s\in\Ct_+$, the Laplace transform of the vector field $\bm{u}(t)$ is defined as:
\ben
\ch{\bm{u}}(s)=\mathscr{L}(\bm{u})(s)=\int_0^{\ify}e^{-st}\bm{u}(t)dt.
\enn
The Fourier transform of $\phi(\wit{x},x_3)$ is normalized as follows:
\ben
\wih{\phi}(\xi,x_3)=\mathscr{F}(\phi)(\xi,x_3)
=\frac{1}{2\pi}\int_{\R^2}e^{-i\wit{x}\cdot\xi}\phi(\wit{x},x_3)d\wit{x},\;\;\xi\in\R^2
\enn
and the inverse Fourier transform of $\wih{\phi}(\xi)$ is
\ben
\phi(\wit{x},x_3)=\mathscr{F}^{-1}(\wih{\phi})(\wit{x},x_3)
=\frac{1}{2\pi}\int_{\R^2}e^{i\wit{x}\cdot\xi}\wih{\phi}(\xi,x_3)d\xi.
\enn
Some related properties on the Laplace transform and its inversion are summarized as
\begin{align}\label{A.1}
    &\mathscr{L}(\frac{d\bm{u}}{dt})(s)=s\mathscr{L}(\bm{u})(s)-\bm{u}(0),\\ \label{A.2}
    &\mathscr{L}(\frac{d^2\bm{u}}{dt^2})(s)
    =s^2\mathscr{L}(\bm{u})(s)-s\bm{u}(0)-\frac{d\bm{u}}{dt}(0),\\ \label{A.3}
    &\mathscr{L}\Big(\int_0^t\bm{u}(\tau)d\tau\Big)(s)=s^{-1}\mathscr{L}(\bm{u})(s),
\end{align}
which can be easily verified from the integration by parts.

Next, we present the relation between Laplace and Fourier transform. According to the
definition on the Fourier transform, it holds
\ben
\sqrt{2\pi}\mathscr{F}(\bm{u}(\cdot)e^{-s_1\cdot})=\int_{-\ify}^{+\ify}\bm{u}(t)e^{-s_1t}e^{-is_2t}dt
=\int_{0}^{\ify}\bm{u}(t)e^{-(s_1+is_2)t}dt=\mathscr{L}(\bm{u})(s_1+is_2).
\enn
We can verify from the formula of the inverse Fourier transform that
\ben
\bm{u}(t)e^{-s_1t}=\frac{1}{\sqrt{2\pi}}\mathscr{F}^{-1}\{\mathscr{F}(\bm{u}(\cdot)e^{-s_1\cdot})\}
=\frac{1}{\sqrt{2\pi}}\mathscr{F}^{-1}\Big(\mathscr{L}(\bm{u}(s_1+is_2))\Big),
\enn
which implies that
\be\label{A.4}
    \bm{u}(t)=\frac{1}{\sqrt{2\pi}}\mathscr{F}^{-1}\Big(e^{s_1t}\mathscr{L}(\bm{u}(s_1+is_2))\Big).
\en
where $\mathscr{F}^{-1}$ denotes the inverse Fourier transform with respect to $s_2$.

By (\ref{A.4}), the Plancherel or Parseval identity for the Laplace transform can be obtained
(see \cite[(2.46)]{Cohen2007}).
\begin{lem}[Parseval identity]\label{lemA.1}
    If $\ch{\bm{u}}=\mathscr{L}(\bm{u})$ and $\ch{\bm{v}}=\mathscr{L}(\bm{v})$, then
    \be
        \frac{1}{2\pi}\int_{-\ify}^{\ify}\ch{\bm{u}}(s)\cdot\ch{\bm{v}}(s)ds_2
        =\int_0^{\ify}e^{-2s_1t}\bm{u}(t)\cdot\bm{v}(t)dt.\label{A.5}
    \en
    for all $s_1>\la$ where $\la$ is the abscissa of convergence for the Laplace transform of $\bm{u}$ and $\bm{v}$.
\end{lem}

\begin{lem}{\rm (\cite[Theorem 43.1]{treves1975})}\label{lemA.2}
    Let $\ch{\bm{\om}}(s)$ denotes a holomorphic function in the half plane $s_1>\sig_0$,
    valued in the Banach space $\E$. The following statements are equivalent:
    \begin{enumerate}[1.]
      \item there is a distribution $\om\in\mathcal{D}_+^{'}(\E)$ whose Laplace transform is equal to
      $\ch{\bm{\om}}(s)$, where $\mathcal{D}_+^{'}(\E)$ is the space of distributions on the real line
      which vanish identically in the open negative half line;
      \item there is a $\sig_1$ with $\sig_0\leq\sig_1<\ify$ and an integer $m\geq0$ such that for all
      complex numbers $s$ with $s_1>\sig_1$ , it holds that $\|\ch{\bm{\om}}(s)\|_{\E} \lesssim(1+|s|)^m$.
    \end{enumerate}
\end{lem}

\section{Functional spaces}\label{ap2}
In this subsection, we give a brief summary of some fundamental functional spaces. For a bounded
Lipschitz domain $D\in\R^3$ with unit outward normal vector $\bm\nu$ on its boundary $\Sig$, we set
\ben
  H(\curl,D):=\{\bm \om\in L^2(D)^3:\ \na\ti\bm \om\in L^2(D)^3\},
\enn
which is clearly a Hilbert space equipped with the norm
\ben
  \V\bm \om\V_{H(\curl,D)}=\left(\V\bm \om\V^2_{L^2(D)^3}+\V\na\ti\bm \om\V^2_{L^2(D)^3}\right)^{1/2}.
\enn

From \cite{Buffa2002}, we define the bounded surjective trace operator $\g$,
 tangential trace operator $\g_t$ and tangential projection operator $\g_T$ by
\begin{align*}
\g&:\;H^1(D)\rightarrow H^{1/2}(\Sigma),\quad\g\varphi=\varphi\quad\on\;\;\Sigma,\\
\g_t&:H^1(D)^3\ra L_t^2(\Sig)^3,\quad\gamma_t\bm \om=\bm \om\ti\bm\nu\quad\on\;\Sigma,\\
\g_T&:H^1(D)^3\ra L_t^2(\Sig)^3,\quad\gamma_T\bm \om=\bm\nu\ti(\bm\om\ti\bm\nu)\quad\on\;\Sigma,
\end{align*}
where $L_t^2(\Sig)^3:=\{\bm\om\in L^2(\Sig)^3:\;\bm\om\cdot\bm\nu=0\}$ and denote by
$\bm\om_{\Sigma}=\bm\nu\ti(\bm\om\ti\bm\nu)$ the tangential component of $\bm\om$ on $\Sigma$.
In fact, the range of $\g_t$ and $\g_T$
\begin{align*}
H_{ \|}^{1 / 2}(\Sig) :=\left\{\bm\xi \in L^2_{t}(\Sig)^3 : \bm\xi=\g_{T} \bm\om,\right.& \bm\om \in H^{1}(D)^3\}, \\
H_{\perp}^{1 / 2}(\Sig) :=\left\{\bm\xi \in L^2_{t}(\Sig)^3 : \bm\xi=\gamma_{t} \bm\om,\right.& \bm\om \in H^{1}(D)^3\},
\end{align*}
are dense in $L^2_t(\Sig)^3$, and $\g_t:H^1(D)^3\ra H_{\perp}^{1/2}(\Sig)$, $\g_T:
H^1(D)^3\ra H_{\parallel}^{1/2}(\Sig)$ are bounded and surjective operators. The dual spaces of 
$H_{\perp}^{1/2}(\Sig)$ and $H_{\parallel}^{1/2}(\Sig)$ with respect to the pivot space $L_t^2(\Sig)^3$
are denoted by $H_{\perp}^{-1/2}(\Sig)$ and $H_{\parallel}^{-1/2}(\Sig)$, respectively.
In this paper, we will also use the notion $\g_t\bm\phi\;(\text{or}\;\g_T\bm\phi)$ for the composite operator $\g_t\circ\g^{-1}\bm\phi\;(\text{or}\;\g_T\circ\g^{-1}\bm\phi)$.
According to \cite[Theorem 4.1]{Buffa2002}, the definitions of $\g_t$ and $\g_T$ can be extended into $H(\curl,D)$.
\begin{lem}\label{trace}
  \ben
    H^{-1/2}(\Dive,\Sig):=\left\{\bm\la\in H^{-1/2}_{\parallel}(\Sig):\;\Dive\;\bm\la\in H^{-1/2}(\Sig)\right\}
  \enn
  and
  \ben
    H^{-1/2}(\Curl,\Sig):=\left\{\bm\la\in H^{-1/2}_{\perp}(\Sig):\;\Curl\;\bm\la\in H^{-1/2}(\Sig)\right\}.
  \enn
  The operators $\g_t:H(\curl,D)\ra H^{-1/2}(\Dive,\Sig)$ and $\g_T:H(\curl,D)\ra H^{-1/2}(\Curl,\Sig)$
  are linear, continuous, and surjective. Moreover, the
  $L^2_t(\Sig)$-inner product can be extended to define a duality product $\langle\cdot,\cdot\rangle_{\Sig}$ between
  the spaces $H^{-1/2}(\Dive,\Sig)$ and $H^{-1/2}(\Curl,\Sig)$.
\end{lem}

We refer to \cite{Buffa2002} for the detailed definitions of the surface divergence and surface scalar curl operators
$\Dive$ and $\Curl$ in lemma\ref{trace}. In addition, the dual pair $H^{-1/2}(\Dive,\Sig)$ and $H^{-1/2}(\Curl,\Sig)$
satisfy the following vector integration by parts
\be\label{Green's}
  \int_{D}(\bm\om\cdot\na\ti\bm v-\na\ti\bm\om\cdot\bm v)dx= \langle\g_t\bm\om,\g_T\bm v\rangle_{\Sig}
  \quad\quad\forall\;\bm\om,\bm v\in H(\curl,D).
\en

For a finite strip domain $\Om_h$, the definition of Sobolev space $H(\curl,\Om_h)$ can be found in \cite{GL2017,LWZ2011}.
Denote by $C_{\wit x}^{\ify}$ the linear space of infinitely differentiable functions with compact support with respect
to the variable $\wit x$ on $\Om_h$. According to the dense argument of $C_{\wit x}^{\ify}(\Om_h)^3$ in $H(\curl,\Om_h)$
(see \cite[Lemma 2.2]{LWZ2011}), one may only need to consider the proof in $C_{\wit x}^{\ify}(\Om_h)^3$ and then extend them
by limiting argument to more general functions in $H(\curl,\Om_h)$. Therefore, the boundary integrals only on $\Ga_{h_j}$
and $\Ga$ need to be considered when formulating the variational problems in $\Om_h$.

For a smooth vector $\bm\om = (\om_1, \om_2, \om_3)^{\top}$ defined on $\Ga_{h_j}$, denote by 
\ben
	\dive_{\Ga_{h_j}}\bm\om = \pa_{x_1} \om_1 + \pa_{x_2} \om_2,\quad \curl_{\Ga_{h_j}} \bm\om = \pa_{x_1} \om_2 - \pa_{x_2} \om_1
\enn
the surface divergence and the surface scalar curl, respectively.
Now we introduce two vector trace spaces on the planar surface:
\ben
&H^{-1/2}(\curl,\Ga_{h_j}):=\big\{\bm\om\in H^{-1/2}(\Ga_{h_j})^3:\ \om_3=0,
\ \curl_{\Ga_{h_j}}\bm\om\in H^{-1/2}(\Ga_{h_j})\big\},\\
&H^{-1/2}(\dive,\Ga_{h_j}):=\big\{\bm\om\in H^{-1/2}(\Ga_{h_j})^3:\ \om_3=0,
\ \dive_{\Ga_{h_j}}\bm\om\in H^{-1/2}(\Ga_{h_j})\big\},
\enn
which are equipped with the norm defined by the Fourier transform:
\ben
&\|\bm\om\|_{H^{-1/2}(\curl,\Ga_{h_j})}=\Big(\int_{\R^2}(1+|\xi|^2)^{-1/2}(|\wih\om_1|^2
+|\wih\om_2|^2+|\xi_1\wih\om_2-\xi_2\wih\om_1|^2)d\xi\Big)^{1/2},\\
&\|\bm\om\|_{H^{-1/2}(\dive,\Ga_{h_j})}=\Big(\int_{\R^2}(1+|\xi|^2)^{-1/2}(|\wih\om_1|^2
+|\wih\om_2|^2+|\xi_1\wih\om_1+\xi_2\wih\om_2|^2)d\xi\Big)^{1/2}.
\enn
The following two lemmas about the duality between the spaces $H^{-1/2}(\curl,\Ga_{h_j})$ and
$H^{-1/2}(\dive,\Ga_{h_j})$ and the trace regularity in $H(\curl,\Om_h)$  can be found the proofs
in \cite[Lemma 2.3, Lemma 2.4]{LWZ2011}.

\begin{lem}
	The spaces $H^{-1/2}(\dive,\Ga_{h_j})$ and $H^{-1/2}(\curl,\Ga_{h_j})$ are mutually adjoint
	with respect to the scalar product in $L^2(\Ga_{h_j})^3$ defined by
	\ben
	\langle\bm\om,\bm v\rangle_{\Ga_{h_j}}=\int_{\Ga_{h_j}}\bm\om\cdot\ov{\bm v}d\g
	=\int_{\R^2}(\wih \om_1\ov{\wih v}_1+\wih \om_2\ov{\wih v}_2)d\xi.
	\enn
\end{lem}

\begin{lem}\label{lemB.2}
	Let $\eta=\max\{\sqrt{1+(h_1-h_2)^{-1}},\sqrt{2}\}$. We have the estimate
	\ben
	\|\bm\om\|_{H^{-1/2}(\curl,\Ga_{h_j})}\leq \eta\|\bm\om\|_{H(\curl,\Om)},\ \forall\;\bm\om\in H(\curl,\Om_h).
	\enn
\end{lem}

\end{appendix}

\section*{Acknowledgements}

The work was partially supported by the National Natural Science
Foundation of China grants 11771349 and 91630309, 
and the National Research Foundation of Korea (NRF-2020R1I1A1A01073356).

\end{document}